\documentclass[11pt]{amsart}
\usepackage{comment}
\usepackage{amssymb}
\usepackage{soul}
\usepackage{xfrac}
\usepackage{bm}
\usepackage[dvipsnames]{xcolor}

 \usepackage{color}

\usepackage[margin=1.2in]{geometry}
\usepackage{amsmath}
\usepackage{amsthm}
\usepackage{hyperref}
\usepackage{soul}
\usepackage{amsfonts}
\usepackage{color}
\usepackage{amssymb}
\usepackage{amscd}%%%added by Jan for V. Kac 4/23/07
\numberwithin{equation}{section}
\newtheorem{thm}{Theorem}[section]
\newtheorem{prop}[thm]{Proposition}
\newtheorem{coro}[thm]{Corollary}
\newtheorem{conj}[thm]{Conjecture}
\newtheorem{lem}[thm]{Lemma}
\newtheorem{rem}[thm]{Remark}
\newtheorem{definition}[thm]{Definition}

\newcommand\half{\frac{1}{2}}

\newcommand\be{\beta}

\newcommand\g{\mathfrak g}
\newcommand\ga{\widehat{\mathfrak g}}
\newcommand\h{\mathfrak h}
\newcommand\ha{\widehat{\mathfrak h}}

\newcommand\D{\Delta}
\renewcommand\l{\lambda}
\newcommand\Dp{\Delta^+}

\renewcommand\d{\delta}

\renewcommand\a{\alpha}
\renewcommand\aa{\mathfrak a}

\newcommand{\Z}{\mathbb Z}

\newcommand\ganz{\mathbb Z}

\renewcommand\L{\Lambda}
\renewcommand\aa{\mathfrak a}

\newcommand\C{\mathbb C}

\renewcommand\ha{\widehat{\mathfrak h}}

\newcommand{\ZZ}{\mathbb{Z}}

\newcommand{\vac}{{\bf 1}}

\newcommand{\bea}{\begin{eqnarray}}
\newcommand{\eea}{\end{eqnarray}}

\newcommand{\sKlf}{{KL^{fin}_k}}
 \newcommand{\Kl}{{KL_k}}
\newtheorem{de}[thm]{Definition}

\begin{document}
 
\title[ On the semisimplicity of the category $KL_k$ for affine  Lie superalgebras]{ On the semisimplicity of the category $KL_k$ for affine  Lie superalgebras }

%\authors{  D. Adamovi\' c, P. Moseneder-Frajria, P. Papi,  and }

\author[Adamovi\'c, M\"oseneder, Papi]{Dra{\v z}en~Adamovi\'c}
\author[]{Pierluigi M\"oseneder Frajria}
\author[]{Paolo  Papi}
 \begin{abstract}
We study  the semisimplicity of the category $KL_k$ for affine  Lie superalgebras and provide  a super analog  of certain results from \cite{AKMPPIMRN}.  Let $\sKlf$ be the subcategory of $KL_k$ consisting of ordinary modules on which a Cartan subalgebra acts semisimply. We prove that $\sKlf$ is semisimple  when 1) $k$ is a collapsing level, 2) $W_k(\g, \theta)$ is rational, 3) $W_k(\g, \theta)$  is semisimple in a certain category. The analysis of the semisimplicity of $\Kl$ is subtler than in the Lie algebra case, since in super case $KL_k$ can contain indecomposable  modules. We are able to prove that in many cases  when $\sKlf$ is semisimple we indeed have $\sKlf=\Kl$, which therefore excludes indecomposable and logarithmic modules in $\Kl$. In these cases we are able to prove that  there is a conformal embedding $W \hookrightarrow V_k(\g)$ with  $W$  semisimple (see Section \ref{9}). In particular, we prove the semisimplicity of $\Kl$ for $\g=sl(2\vert 1)$ and $k = -\frac{m+1}{m+2}$, $m \in {\Z}_{\ge 0}$. For $\g =sl(m \vert 1)$, we prove that $\Kl$ is semisimple for $k=-1$, but for $k$ a positive integer we show that it is not semisimple by constructing indecomposable highest weight modules in $\sKlf$.
 \end{abstract}
\date{\today}

\maketitle

 \section{Introduction}
 %\da{slightly edited}
 
The aim of this paper is to extend results from \cite{AKMPPIMRN} to  the  super case.
In \cite{AKMPPIMRN}  we proved semisimplicity of the category of ordinary modules $\Kl$ for a simple affine vertex algebra $V_k(\g)$ when $\g$ is a Lie algebra and one of the following assumptions holds:
\begin{enumerate}
\item[(L-1)] $k$ is collapsing level, i.e.,  $W_k(\g, \theta) $  collapses to its affine subalgebra;
\item[(L-2)] $W_k(\g, \theta) $ is rational;
\item[(L-3)] $W_k(\g, \theta) $ is semisimple in a category of ordinary modules.
\end{enumerate}
Here $W_k(\g, \theta) $ is the simple affine $W$-algebra attached to $\g$ and a minimal nilpotent element \cite{KW2}. It is   natural   to look for an extension of these results when $\g$ is a  Lie superalgebra.

Our paper \cite{AKMPPIMRN} has attracted a lot of interest in various direction. Here we mention two interesting applications/generalizations:

\begin{itemize}
\item T. Creutzig and J. Yang showed in  \cite{CY}  that at all levels investigated in \cite{AKMPPIMRN} there is a braided tensor category structure. It is interesting that they also use our previous work on the decomposition of conformal embeddings to prove rigidity of tensor categories.  
\item T. Arakawa, J. van Ekeren and A. Moreau in \cite{ArEM} have constructed a large new family of collapsing levels which are admissible.   
\end{itemize}

 %Collapsing levels for Lie superalgebras were described in \cite{AMPP-AIM}. 
 In order to apply the results from \cite{AKMPPIMRN}, we need to extend certain structural resuls on self-extension of irreducible modules in $\Kl$. A first problem is that  the category $\Kl$ for Lie superalgebras is more complicated than  in the case of Lie algebras. For instance, in the super case, $\Kl$ can have non-semisimple and logarithmic modules. A nice illustration for these phenomena in given by  the Lie superalgebra $\mathfrak{gl}(1\vert 1)$ and its affine vertex algebra  $V_1(\mathfrak{gl}(1\vert 1))$,  which admits highest weight modules whose top components are two dimensional indecomposable 
 $\mathfrak{gl}(1\vert 1)$--modules (cf.  \cite{AdP-2019}, \cite{CMY}). One can also construct logarithmic $\mathfrak{gl}(1\vert 1)$--modules in $\Kl$.
 
 \subsection{Semisimplicity of $\sKlf$.}In order to extend directly our methods from \cite{AKMPPIMRN}, it seems that we have one natural choice. We can consider the smaller categories  
 $KL_k^{ss}$ (resp. $\sKlf$)  consisting of modules from $\Kl$ on which   $L_{\g}(0)$ acts semisimply (resp. a Cartan subalgebra of the affinization $\widehat{\g}$ of $\g$ acts semisimply), cf. Definition \ref{KLK}.
 We prove in  Theorem \ref{TT} a super analog of  \cite[Theorem 5.5]{AKMPPIMRN}:
 \begin{itemize}
 \item Let $\g$ be a  basic Lie superalgebra. If  every highest weight $V_k(\g)$--module  in $\sKlf$  is irreducible, then  the category $\sKlf$ is semisimple.
\end{itemize}

\par \noindent

 With this modification we can prove the semisimplicity of $\sKlf$ in the following cases
 \begin{itemize}
  \item[(S-2)] $W_k(\g, \theta) $ is rational.
 \item[(S-3)] $W_k(\g, \theta) $ is semisimple in the category of ordinary modules.
\end{itemize}
Regarding the super analogue of condition (L-1), i.e., $k$ is a collapsing level, we introduce the notion of {\it collapsing chain} (cf. Definition \ref{cc}) and in many cases we reduce to prove semisimplicity of $\sKlf$ by looking at conditions  (S-2) or  (S-3)  for  explicitly determined subalgebras $\g_n$ of $\g$ (see Theorem \ref{maincoll}).
 Collapsing levels for Lie superalgebras were classified in \cite{AKMPPIMRN}. So our results immediately gives semisimplicity of $\sKlf$ for these levels.    A comprehensive list of all the cases covered by the above conditions is given in Corollary \ref{list}.

 Let us mention some cases of rationality:
 
 \begin{itemize}
 
 \item  $\g = osp(1 \vert 2)$. Then the minimal $W$-algebra $W_k(\g, \theta)$ is the $N=1$ super-conformal algebra which is rational for certain $k$.

 \item $\g = sl(2 \vert 1)$.  Then for  $k = - \frac{m+1}{m+2},\,m\in \mathbb Z_{\ge 0}$, the algebra $W_k(\g, \theta)$ is isomorphic to $N=2$ superconformal vertex algebra at central charge $c=3 m / (m+2)= -3 (2k +1)$, which is rational by \cite{A-2001}.

 \end{itemize}

Next we have interesting cases when $W_k(\g, \theta) $ is semisimple in a certain category. 

\begin{itemize}
\item $\g = \mathfrak{psl}(2 \vert 2)$ and conformal level $k=1/2$. Then  $W_k(\g, \theta)$ is the $N=4$ superconformal vertex algebra with central charge $c=-9$ \cite{A-TG}, which is semisimple in the category of ordinary modules  (cf. Theorem \ref{psl22-n4}).

\item $\g = D(2,1, \alpha)$.  Then for collapsing level $k$ we can have that $W_k(\g, \theta) = V_{k'} (sl(2))$. If $k'$ is a positive integer,  or admissible we conclude that $V_k (\g)$ is semisimple in $KL_k$.
In \cite{AKMPP-JJM} we described a  conformal embedding $V_{k_1} (sl(2)) \otimes V_{k_2} (sl(2)) \hookrightarrow W_k(\g, \theta)$. If both $k_1, k_2$ are admissible, we expect that  then $\sKlf$ is semisimple.
\end{itemize}
 
  More results on the semisimplicity of $KL_k$, regarding  $V_{-\tfrac{n+1}{2}}(C(n+1)), V_{-1}(psl(m|m)), m\geq 3,$ are given in Theorems \ref{newC}, \ref{72} respectively.

  \subsection{When $\bf \Kl = \sKlf $ ?} We have already observed that 
the category $\Kl$ is the most natural choice of category of $V_k(\g)$--modules.  The example of $V_1(\mathfrak{gl}(1\vert 1))$ shows that in general   $\Kl \ne \sKlf $.   
 On the other hand, we can prove equality in some cases.  
  In Proposition \ref{klf=kl} we give  two  sufficient conditions  for this equality to hold.
\begin{itemize}

\item Assume that $\g_{\overline 0}$ is a semisimple Lie algebra. Then
$\sKlf=KL_k ^{ss}$.
 \item There is a  conformal embedding of  $V_{k_1} (\g_{\overline 0}) \hookrightarrow V_k(\g)$ such that the category   $KL_{k_1}$ for $V_{k_1} (\g_{\overline 0})$ is semisimple. 
 \end{itemize}
 These conformal embeddings were classified in  \cite{AMPP-AIM}.
 
 We believe that when  $\sKlf$ is semisimple, we should have  $\Kl = \sKlf $.  Under the assumption that  $\sKlf$ is semisimple, we prove in Theorem   \ref{ext-kl-large} the following result:
  \begin{itemize}
  \item Assume that   for any irreducible    $V_k(\g)$--module $M$   in $\Kl$ we have  
 \begin{equation}\label{Ext-1-2-intr}  \mbox{Ext}^{1} (M_{top}, M_{top}) = \{0\}\end{equation}
  in the category of finite-dimensional $\g$--modules.
    Then $\Kl $ is semisimple and $\sKlf=\Kl$.
 \end{itemize}

 We will prove  that $\Kl=\sKlf$ in  the following cases:
\begin{itemize}  
\item $V_{-\tfrac{1}{2}}(C(n+1))$; the proof (see Theorem \ref{cn1}) relies on results from \cite{AMPP-AIM} and fusion rules arguments.
\item $V_{-1}(sl(m \vert 1))$ (see Theorem \ref{complete-red-slm1}). 
\item[] Since $k=-1$ is a collapsing level, we get that $\sKlf$ is semisimple. Next we refine the classification of irreducible modules in $\Kl$ and prove that  top components of   irreducible modules in $\Kl$ are atypical $\g$--modules.  Then the result of \cite{Ger} on extensions of finite-dimensional $sl(m \vert 1)$--modules implies that there are no self-extensions among irreducible modules in $\Kl$, so the condition  (\ref{Ext-1-2-intr})  is satisfied. Then Theorem   \ref{ext-kl-large}  gives that the larger category $\Kl$ is semisimple.

\item $V_{-(m+1)/ (m+2)}(sl(2 \vert 1))$, $m  \in {\Z}_{\ge 0}$. In this case we prove that the center of $\g_{\overline 0}$  belongs to a regular vertex operator  algebra $D_{m+1, 2}$ from \cite{ACEJM}: see Section \ref{9}. Then    the condition  (\ref{Ext-1-2-intr}) is also satisfied, which gives that in this case  the category $\Kl$ is semisimple.
\end{itemize}
 
 \subsection{Examples when $\Kl$ is not semisimple}
M. Gorelik and V. Serganova in  \cite[Section 5.6.4]{GS} constructed examples of indecomposable weak $V_1(sl(2 \vert 1))$--modules on which the Sugawara operator $L(0)$ does not act semisimply. Their construction uses the theory of Zhu's algebras and a description of the maximal ideal in the universal affine vertex algebra $V^1(sl(2 \vert 1))$. In the present paper we  use a different approach and apply free-field realisation. In Theorem \ref{free-field-indecom}  we construct a highest weight $V_1(sl(m \vert 1))$--module
$ \widetilde W = V_1(sl(m \vert 1)). (a^+ )^{-m} \otimes \vert m>$ which has a proper submodule isomorphic to $V_1(sl(m \vert 1))$.
The module $ \widetilde W$ belongs to $\sKlf$, and therefore $\Kl$ is not semisimple for $k=1$. We extend this example by showing that $\sKlf$ is not semisimple for $k \in {\Z}_{>0}$.   A complete analysis of indecomposable modules will appear in forthcoming papers. 
  \vskip 5mm
  {\bf Acknowledgements.}
  We would like to thank to Maria Gorelik, Victor Kac, Ozren Per\v se, Thomas Creutzig and Veronika Pedi\' c on useful discussions.
   We thank the referee for his/her careful reading of the paper and some very helpful hints.
  
D.A.   is  partially supported   by the
QuantiXLie Centre of Excellence, a project cofinanced
by the Croatian Government and European Union
through the European Regional Development Fund - the
Competitiveness and Cohesion Operational Programme
(KK.01.1.1.01.0004).

 \section{Setup} 
 Let $\g=\g_{\bar 0}\oplus \g_{\bar 1}$   be a basic Lie superalgebra, i.e. a simple Lie superalgebra such that  $\g_{\bar 0}$ is a reductive Lie algebra and there exists an even invariant supersymmetric bilinear form on it (see \cite{K0} for more details and the classification). Choose a Cartan subalgebra $\h$ for $\g_{\bar 0}$  and let $\D=\D_0\cup\D_1$ be the set of roots. Fix a positive system $\Dp$ in $\D$ and choose an even root $\theta$ maximal in $\Dp_0=\Dp\cap\D_0$. We may choose root vectors $e_\theta$ and $e_{-\theta}$ such that 
$$[e_\theta, e_{-\theta}]=x\in \h,\qquad [x,e_{\pm \theta}]=\pm e_{\pm \theta}.$$
Due to the minimality of $-\theta$, the eigenspace decomposition of $ad\,x$ defines a {\it minimal} $\frac{1}{2}\ganz$-grading  (\cite[(5.1)]{KW2}):
\begin{equation}\label{gradazione}
\g=\g_{-1}\oplus\g_{-1/2}\oplus\g_{0}\oplus\g_{1/2}\oplus\g_{1},
\end{equation}
where $\g_{\pm 1}=\C  e_{\pm \theta}$.  
Furthermore, one has
\begin{equation}\label{gnatural}
\g_0=\g^\natural\oplus \C x,\quad\g^\natural=\{a\in\g_0\mid (a|x)=0\}.
\end{equation}
Note that  $\g^\natural$ is the centralizer of the triple $\{f_\theta,x,e_\theta\}$.
We can choose $
\h^\natural=\{h\in\h\mid (h|x)=0\},
$ as a  Cartan subalgebra of the Lie superalgebra $\g^\natural$,  so that $\h=\h^\natural\oplus \C x$.\par
For a given choice of a minimal root $-\theta$, we normalize the invariant bilinear form $( \cdot | \cdot)$ on $\g$ by the condition
\begin{equation}\label{normalized}
(\theta | \theta)=2.
\end{equation}
The dual Coxeter number $h^\vee$ of the pair $(\g, \theta)$  is defined to be   half the eigenvalue of the Casimir operator of $\g$ corresponding to $(\cdot|\cdot)$, normalized  by \eqref{normalized}.
Let $\ga$  be the affinization of  $\g$, i.e.
$$\ga=\C[t,t^{-1}]\otimes\g\oplus\C K\oplus \C d,$$ 
where $d$ acts as $t\tfrac{d}{dt}$, $K$  is central and the bracket on $[\ga,\ga]$ is defined by 
$$[x_{(m)},y_{(n)}]=[x,y]_{(m+n)}+m\d_{m,-n}(x|y)K,$$
where  $x_{(m)}=t^{m}\otimes x,\,x\in\g$. Set $\ha=\h\oplus\C K\oplus \C d$. Write
$\ha^*=\h^*\oplus\C \L_0\oplus\C\d$ where $\L_0(K)=1, \L_0(\h)=0, \L_0(d)=0,  \d(K)=0, \d(\h)=0, \d(d)=1$.
\par
By category $\mathcal O$  for $\widehat{\mathfrak g}$ we mean the set of $\widehat{\mathfrak g}$--modules which are $\ha$-diagonalizable with finite dimensional weight spaces and a finite number of maximal weights. Let  $\mathcal O^k$ be the subcategory of $\ga$-modules in $\mathcal O$ of level $k$  (i.e., $K$ acts as $k\,Id$). \par
Let $\aa$ be a Lie superalgebra equipped with a nondegenerate  invariant supersymmetric bilinear form $B$. The universal affine vertex algebra $V^B(\aa)$ is  the universal enveloping vertex algebra of  the non--linear Lie conformal superalgebra $R=(\C[T]\otimes\aa)$ with $\lambda$-bracket given by
$$
[a_\lambda b]=[a,b]+\lambda B(a,b),\ a,b\in\aa.
$$
In the following, we shall say that a vertex algebra $V$ is an affine vertex algebra if it is a quotient of some $V^B(\aa)$.

If $k\in \C$,  we will write simply $V^k(\g)$ for $V^{k(\cdot |\cdot)}(\g)$. We  will always assume that $k$ is non--critical, i.e. $k\ne - h^\vee$. With this assumption, it is known that $V^k(\g)$ has a unique simple quotient, denoted by   $V_k(\g)$  (see \cite[\S\ 4.7 and Example 4.9b]{K2}). The vertex algebras $V^k(\g), V_k(\g)$ are VOAs with Virasoro vector $L_\g$ given by the Sugawara construction.
\par If $M$ is a restricted  module of level $k$ for $\ga$ then it is a  weak  module for $V^k(\g)$; conversely, letting  $d$ act on 
weak modules by $-L_{\g}(0)$ yields restricted modules for $\ga$.

\begin{definition}\label{KLK} We denote by $KL^B(\g)$ the category of  weak modules for $V^B(\g)$, which 
\begin{enumerate}
\item[(1)] are locally finite  as $\g$-modules; 
\item[(2)] admit a decomposition into generalized eigenspaces for $L_\g(0)$ whose eigenvalues are bounded below.
\end{enumerate}
 We denote by 
\begin{itemize}
\item $KL^B_{fin}(\g)$ the full subcategory of
 modules in $KL^B(\g)$ on which  $\ha$ acts semisimply.
 \item $KL^B _{ss}(\g)$ the full subcategory of
 modules in $KL^B(\g)$ on which  $L_{\g}(0)$ acts semisimply.
 \end{itemize}
  If $B=k(\cdot|\cdot)$ we simply write $KL^k(\g)$, $KL^k _{fin} (\g)$,   $KL^k _{ss} (\g)$. We also denote by $\Kl(\g)$, $\sKlf(\g)$,  $KL_k ^{ss} (\g) $ the full subcategories of $KL^k(\g), KL^k_{fin}(\g),  KL^k _{ss} (\g)$ consisting of the  $V_k(\g)$-modules.
 If $\g$ is clear from the context we omit it in the notation. 
\end{definition}

Let  $V$ be  a conformal  vertex algebra. Denote by  $L$ its conformal vector (with $Y(L,z)=\sum_{n\in\mathbb Z} L(n)z^{-n-2}$). We say that a $V$--module is ordinary if $L(0)$ acts semisimply with finite dimensional eigenspaces. In our settings, let $C_k$ be the category of ordinary  modules.

\begin{rem}
Note that (2)  is exactly the definition of a logarithmic module for a  conformal vertex   algebra.  Condition (1) says that each generalized eigenspace for $L(0)$ is a sum of finite-dimensional $\g$--modules. When $\g$ is a simple Lie algebra, we have that each generalized eigenspace for $L(0)$ is a direct sum of finite-dimensional modules.
In particular, we have $\sKlf=KL_k ^{ss}$.
But when $\g$ is a Lie superalgebra, we really have modules in $\Kl$ and $KL_k^{ss}$ with non--semisimple action of $\g$. One such example is the vertex algebra $V_1(\mathfrak{gl}(1\vert 1))$  and  its modules considered in  \cite{AdP-2019}.  
More examples are given in Section \ref{non-semisimple}.   \end{rem}

\section{$W$-algebras and collapsing levels}
Denote by $W^{k}(\g,\theta)$ the affine $W$--algebra obtained from  $V^{k}(\g)$ by Hamiltonian reduction
relative to the minimal nilpotent element $e_{-\theta}$. More precisely, let 
$M$ be  a restricted $V^{k}(\g)$-module. Consider the complex
\begin{equation}\label{Cm}\mathcal C^M=M\otimes F(A_{ch})\otimes F(A_{ne}),\end{equation}
where $F(A_{ch}), F(A_{ne})$ are fermionic vertex algebras,  defined in \cite{KW2}, attached to the following superspaces.
Denote by $A_{ne}$ the vector superspace $\g_{1/2}$
with the bilinear form 
\begin{equation}
  \label{eq:1.3}
  \langle a , b \rangle_{ne} = (e_{-\theta} | [a,b]) \, .
\end{equation}

 Denote by $A$
(resp. $A^*$) the vector superspace $\g_{1/2}\oplus\g_{1}$ (resp. $(\g_{1/2}\oplus\g_{1})^*$)
with the reversed parity, let $A_{ch} =A \oplus A^*$ and define
an even skew-supersymmetric non-degenerate bilinear form $\langle . \,
, \, . \rangle_{ch}$ on $A_{ch}$ by
\begin{eqnarray}
  \label{eq:1.6}
  \langle A,A \rangle_{ch} &=& 0 = \langle A^* ,A^* \rangle_{ch}
   \, , \, \\
\nonumber
\langle a,b^* \rangle_{ch} &=& -(-1)^{p(a)p(b^*)}
   \langle b^* ,a \rangle_{ch} =b^* (a) \hbox{ for }
   a \in A , b^* \in A^* \, .
\end{eqnarray}
Here and further, $p(a)$ stands for the parity of an
(homogeneous) element of a vector superspace.
Choose a basis $\{
u_{\alpha} \}_{\alpha \in S_j}$ of each $\g_j$ in
\eqref{gradazione}, and let $S=\coprod_{j \in \tfrac{1}{2} \ZZ} S_j$,
$S_+ = \coprod_{j>0} S_j$.  Let $p(\alpha) \in \ZZ / 2\ZZ$ denote the parity of
$u_{\alpha}$, and let $m_{\alpha} =j$
if $\alpha \in S_j$.  Define the structure constants
$c^{\gamma}_{\alpha \beta}$ by $[u_{\alpha}, u_{\beta}] =
\sum_{\gamma} c^{\gamma}_{\alpha \beta} u_{\gamma}$ $(\alpha,
\beta, \gamma \in S)$.  Denote by $\{ \varphi_{\alpha}\}_{\alpha
  \in S_+}$ the  basis of $A$  corresponding to $u_\a$ in the identification $A=\g_{1/2}\oplus \g_1$, and by $\{
\varphi^{\alpha}\}_{\alpha \in S_+}$ the basis of $A^*$ such that
$\langle \varphi_{\alpha}, \varphi^{\beta} \rangle_{ch}
=\delta_{\alpha \beta}$.  Denote by $\{ \Phi_{\alpha}\}_{\alpha
  \in S_{1/2}}$ the corresponding basis of $A_{ne}$, and by $\{
\Phi^{\alpha} \}_{\alpha \in S_{1/2}}$ the dual basis
with respect to $\langle . \, , \, . \rangle_{ne}$, i.e.,~$\langle
\Phi_{\alpha} , \Phi^{\beta} \rangle_{ne} = \delta_{\alpha \beta}$.
 Define
$$H(M)=H^0(\mathcal C^M,d_0)$$
where $d_0$ is defined in \cite{KW2}. Then $$W^{k}(\g,\theta)=H(V^k(\g)).$$
Denote by $W_{k}(\g,\theta)$ the unique simple quotient of $W^{k}(\g,\theta)$. It is proved in \cite{KW2} that  $W^{k}(\g,\theta)$ has a vertex subalgebra isomorphic to $V^{\beta_k}(\g^\natural)$, where 
$$\beta_k(a,b)=\tfrac{1}{4}\left((k+h^\vee/2)(a|b)-\tfrac{1}{4}\kappa_0(a,b)\right),\ a,b\in \g^\natural.$$
and $\kappa_0$ is the Killing form of $\g_0$. 
 Let $\mathcal{V}_{k}(\g^\natural)$ be the image of $V^{\beta_k}(\g^\natural)$ in  ${W}_{k}(\g, \theta)$.
\begin{de}\label{CL} If ${W}_{k}(\g, \theta)=\mathcal{V}_{k}(\g^\natural)$, we say that $k$ is  a {\sl collapsing level}.\end{de}

\begin{thm}\cite[Theorem 3.3]{AKMPP-JA} Let $p(k)$ be the polynomial listed in Table 1 below. Then
 $k$  is a collapsing level  if and only if
 $p(k)=0$. \end{thm}

{\small

\begin{table}{}

\caption{}
\vskip 5pt
\begin{tabular}{c|c}
$\g$&$p(k)$\\
\hline
$sl(m|n)$, $n\ne m$&$(k+1) (k+(m-n)/2)$\\
\hline
$psl(m|m)$&$ k (k+1)$\\
\hline
$osp(m|n)$&$(k+2) (k+(m-n-4)/2)$\\
\hline
$spo(n|m)$&$(k+1/2) (k+(n-m+4)/4)$\\
\hline
$D(2,1;a)$&$(k-a)(k+1+a)$\\
\hline
$F(4)$, $\g^\natural=so(7)$ & $(k+2/3)(k-2/3)$\\
\hline
$F(4)$, $\g^\natural=D(2,1;2)$ & $(k+3/2)(k+1)$\\
\hline $G(3)$, $\g^\natural=G_2$ & $(k-1/2)(k+3/4)$ \\
\hline$G(3)$, $\g^\natural=osp(3|2)$ & $(k+2/3)(k+4/3)$
\end{tabular}
\vskip10pt 
\noindent Note: when writing $osp(m|n)=spo(n|m)$ we adopt the conventions of \cite[2.1.2]{K1}, so that $n$ is even.
\end{table}
} 
 
\section{Results in $KL_k^{fin}$}
\label{sect-4}
For $\l \in \ha^*$, denote by $L(\l)$ the irreducible $\ga$-module of highest weight $\l$.

\begin{lem}\label{sl} $Ext^1_{\mathcal O^k}(L(\l),L(\l))=0$.
\end{lem}
\begin{proof}
Suppose that there is an extension
\begin{equation}\label{s1} 0\to L(\l)^{\,\,\,\,\,\,h}{\!\!\!\!\!\!}\to N^{\,\,\,\,\,\,f}{\!\!\!\!\!\!}\to L(\l)\to 0.\end{equation}
Since $\ha$ acts diagonally, this implies that 
\begin{equation}\label{s2}  0\to L(\l)_\l\to N_\l\to L(\l)_\l\to 0\end{equation}
splits. We now prove that  \eqref{s1} splits too. Indeed, let $g_\l:L(\l)_\l\to N_\l$ be a section. Let $M(\l)$ be the Verma module and $\pi:M(\l)\to L(\l)
$ be the canonical projection. Let $\eta: M(\l)\to N$ be the unique map such that $\eta(1)=g_\l(\pi(1))$. Remark that  obviously,  $f(\eta(Ker\,\pi))=0$, so 
$\eta(Ker\,\pi)\subset	 h(L(\l))$. Since $\eta(Ker\,\pi)_\l=0$, we have  $\eta(Ker\,\pi)=0$. Define $g:L(\l)\to N$ by setting $g(\pi(v))=\eta(v)$. It is easy to verify that 
$g$ is a well defined section which splits \eqref{s1}. \end{proof}
\begin{prop}\label{inO} Suppose that $M\in KL^{fin}_k$ is finitely generated. Then $M\in\mathcal O^k$. 
\end{prop}
\begin{proof} By assumption $\ha$ acts semisimply.  Let $\ga_+=\g\oplus t\C[t]\g$. 
Let $\{m_1,\ldots,m_k\}$ be a set of generators for $M$. By the finiteness of the $\g$-action and since the conformal weights are bounded below, $M'=U(\ga_+)(\sum_{i=1}^k\C m_i)$ is finite dimensional, in particular, it has only a finite number of weights.
Let $\ga_-=t^{-1}\C[t^{-1}]\g$, so that  $U(\ga)=U(\ga_-)\otimes U(\ga_+)$ and $M=U(\ga_-)M'$, hence $M$ has a finite number of maximal weights.
Since
$$M_\l=\sum_\nu U(\ga_-)_{\l-\nu} M'_\nu,$$
 we see that the $\ha$-eigenspaces are finite dimensional.
\end{proof}
\begin{thm}\label{TT} Let $\g$ be a  basic Lie superalgebra. If  every highest weight $V_k(\g)$--module  in $KL^{fin}_k$  is irreducible, then  the category $KL^{fin}_k$ is semisimple.
\end{thm}
\begin{proof}  Let $M$ be a module in $KL^{fin}_k$.  If $M$ is finitely generated, then, by Proposition \ref{inO}, we have that  $M\in \mathcal O^k$.  Since Lemma \ref{sl} holds, we can then use the argument in \cite[Theorem 5.5]{AKMPPIMRN} to prove the complete reducibility of $M$. As in  {\it loc. cit.}, the case when $M$ is not finitely generated 
can be reduced to the finitely generated case.
 \end{proof}

If $V$ is a vertex algebra, we say that a $V$--module is ordinary if $L(0)$ acts semisimply with finite dimensional eigenspaces. Recall that we denote by $C_k$ the category of ordinary  modules  in $KL_k$.
 
\begin{coro} Assume that  $\g_{\bar 0}$ is semisimple. If  every highest weight $V_k(\g)$--module  in  $C_k$ is irreducible, then $C_k$ is semisimple. 
\end{coro}
\begin{proof}If $M$ is a highest weight module in $KL^{fin}_k$, then, by Proposition \ref{inO}, $M\in C_k$, thus $M$ is irreducible. By Theorem \ref{TT}, $KL_k^{fin}$ is semisimple. Let $M$ be a module in $C_k$. Since $\h\subset \g_{\bar 0}$ is a Cartan subalgebra of the semisimple Lie algebra $\g_{\bar 0}$, and $\h, L(0)$ commute, then $\ha$  acts semisimply on $M$, thus $M\in KL_k^{fin}$ and therefore $M$ is completely reducible.
\end{proof}
Recall from \cite[Lemma 5.6]{AKMPPIMRN} the following result.
\begin{lem}  \label{kriterij} Let $k \in {\mathbb Q} \setminus {\Z_{\ge 0}}$.
 Assume that $H(U)$ is an irreducible, non-zero $W_k(\g,\theta)$--module  for every non-zero highest weight  $V_k(\g)$--module  $U$ from  the category  $KL^{fin}_k$. Then every highest weight $V_k(\g)$--module in $KL^{fin}_k$  is irreducible.
  \end{lem}
  
\subsection{Rational case}
\begin{thm}\label{T4} Let $\g$ be a  basic Lie superalgebra and  $k\notin \mathbb Z_{\ge 0}$. If $W_{k}(\g,\theta)$ is rational then $KL_{k}^{fin}(\g)$ is semisimple.
 \end{thm}
\begin{proof} By Theorem \ref{TT} it suffices to prove that every highest weight $V_k(\g)$--module  in $KL^{fin}_k$  is irreducible. Let $U\ne 0$ be such a module. Then $H(U)$ is a highest weight module for $W_{k}(\g, \theta)$, which is nonzero since  $k\notin \mathbb Z_{\ge 0}$ . If $W_k(\g,\theta)$ is rational then $H(U)$ is irreducible, hence $U$ is irreducible by Lemma \ref{kriterij}.
\end{proof}
\subsection{Collapsing case}
 
 \begin{definition}\label{cc} Write $(\g_1,k_1)\triangleright (\g_2,k_2)$ if $H(V_{k_1}(\g_1,\theta))=V_{k_2}(\g_2)$. We call a sequence $(\g,k)\triangleright(\g_1,k_1)\triangleright\ldots  \triangleright(\g_n,k_n)$  a collapsing chain for 
$(\g,k)$.
 \end{definition}
 
 \noindent It follows from \cite[Main Theorem]{Araduke}  that if $(\g_1,k_1)\triangleright (\g_2,k_2)$, then $k_1\notin \mathbb Z_{\geq 0}$.
 \begin{prop}\label{nuovofondamentale}
\label{KtoK} Assume that $H(V_k(\g))=V_{k'}(\g^\natural)$. If $M$ is a highest weight module in $KL_k$, then  $H(M)\in KL^{fin}_{k'}$.
\end{prop}
\begin{proof} By \cite{KW2}, we know that $H(M)$ is a highest weight module, in particular  $\widehat\h^\natural$ acts semisimply and the $L_{\g^\natural}(0)$-eigenvalues are bounded below. It remains only  to show that the action of $\g^\natural$ on the complex $\mathcal C^M$ \eqref{Cm} is locally finite. If $v\in \g^\natural$ the the action of $v$ on the complex is given by  $J^{\{v\}}_0$, where
\begin{equation}\label{actionJ}J^{\{v\}}=v+\sum_{\a,\be\in S_+}(-1)^{p(\a)} c_{\a\be}(v):\varphi_\a\varphi^\be: +(-1)^{p(v)/2}\sum_{\a\in S_{1/2}}:\Phi^\a\Phi_{[u_\a,v]} :.
\end{equation}
 $J^{\{v\}}(0)$ acts on the tensor product $M \otimes F$, where $F=F(A_{ch})\otimes F(A_{ne})$ is a Clifford vertex algebra. The action of $v(0)$ on $M$ is locally finite, since the action of $\g$ is locally finite. More precisely, $M$ admits a ${\Z}_{\ge 0}$ gradation:
$$ M = \bigoplus_{\ell \in   {\Z}_{\ge  0}  } M_{\ell}, \quad \g \ \mbox{acts locally finite on} \ M_{\ell}.$$
 On the other hand $F$ admits a Virasoro vector $L_F$ and the corresponding eigenvalue decomposition  
$$ F = \bigoplus_{\ell \in   \tfrac{1}{2}\Z_{\ge  0}  } F_{\ell}, $$
has finite-dimensional eigenspaces. By \eqref{actionJ}, $J^{\{v\}}-v$ is primary of conformal weight $1$, hence $J^{\{v\}}(0)-v(0)$ commutes with $L_F(0)$. In particular it stabilizes the eigenspaces of  $L_F$, so that we have  a gradation:
$$  M \otimes F = \bigoplus_{\ell \in {\half\Z}_{\ge 0}  }  (M \otimes F)_{\ell}, $$
where $$(M \otimes F)_{\ell} = \bigoplus_{i = 0} ^{\ell}  M_{i} \otimes F_{\ell-i}$$
 and each $(M \otimes F)_{\ell}$ is a $\g ^{\natural}$--module. 
 Since each $M_{i} $ is $\g$--locally finite and  $F_{j}$ is finite-dimensional, we conclude that $(M \otimes F)_{\ell}$ is $\g^{\natural}$--locally finite. The claim follows.

\end{proof}

\begin{thm}\label{maincoll} Let  $(\g,k)\triangleright(\g_1,k_1)\triangleright\ldots  \triangleright(\g_n,k_n)$ be a collapsing chain.  Assume that $KL^{fin}_{k_n}(\g_n)$ is  semisimple. Then $KL^{fin}_k$ is semisimple. In particular this happens when $V_{k_n}(\g_n)$ is  rational or admissible.
\end{thm}
\begin{proof} We proceed by induction on the length of the collapsing chain. The base case $n=0$ is obvious. Now assume $n>0$. 
  First remark that every highest weight module $U$  in $KL^{fin}_k$ is irreducible. Indeed if $U$ is such a module, then $H(U)$ is a  $V_{k_1}(\g_1)$--module    in $KL^{fin}_{k_1}$ by Proposition \ref{nuovofondamentale} and non-zero and highest weight by \cite{Araduke}, \cite{KW2}.
  By induction  $KL^{fin}_{k_1}$ is semisimple, hence $H(U)$ is irreducible. By Lemma \ref{kriterij} $U$ is irreducible, so we are in the hypothesis of Theorem \ref{TT} and therefore we can conclude that $KL^{fin}_k$ is semisimple. In particular, if $V_{k_n}(\g_n)$ is  rational or admissible,
then $KL^{fin}_{k_n}(\g_n)$ is semisimple (by definition in the rational case, by  Arakawa's Main  Theorem from 
\cite{Ar4} in the admissible case).
\end{proof}

\begin{lem}\label{IMRNso} There is a collapsing chain $(so(m),\tfrac{4-m}{2})\triangleright \ldots  \triangleright(\g',k')$
with 
\begin{equation}\label{so}
(\g',k')=\begin{cases}\C\quad&\text{if $m\equiv 0, 1 \mod 4$},\\
M(1)\quad&\text{if $m\equiv 2 \mod 4$},\\
(sl(2),1)\quad&\text{if $m\equiv 3 \mod 4$.}
\end{cases}\end{equation}\end{lem}
\begin{proof} Using Table 2 below, one sees  that for $m\gg 0$
$$H(V_{\tfrac{4-m}{2}}(so(m)))=V_{\tfrac{8-m}{2}}(so(m-4)).$$
Since $\tfrac{8-m}{2}$ is  again a collapsing level for $so(m-4)$, by  induction on $m$ we are reduced to the cases when $m=4,5,6,7$, where one concludes using  Table 2  once again.
\end{proof}

\begin{prop} The following are  collapsing chains
\begin{align}
&\label{a}(F(4),-1)\triangleright (D(2,1;2),1/2)= (D(2,1;1/2),1/2)\triangleright (sl(2),-4/3)\\ 
&\label{b}(F(4),2/3)\triangleright (so(7),-2)\triangleright (sl(2),-1/2) \,\,\,\\\
&\label{c}(G(3),1/2)\triangleright (G_2,-5/3)\triangleright  \mathbb C\,\,\\
&\label{l}(spo(2|3),-3/4)\triangleright (sl(2),1)\\
&\label{m}(spo(2|1),-5/4)\triangleright\mathbb C\\
&\label{e}m\ne n+1, n\geq 2,\,(spo(n|m),-1/2)\triangleright  \mathbb C\\
&\label{f} 
m \text{ odd}: \\
&(spo(n|m), \tfrac{m-n-4}{4})\triangleright  (spo(n-2|m),\tfrac{m-n-2}{4})\triangleright\ldots
(spo(2|m),\tfrac{m-6}{4})\triangleright  (so(m),\tfrac{4-m}{2}) \triangleright  \eqref{so}\notag\\
&\label{g}(osp(n+8|n),-2)\triangleright  \C\\
&\label{h} m\ne n+4, n+8, m\ge 4: \,\,(osp(m|n),-2)\triangleright  (sl(2),\tfrac{m-n-8}{2})\\
&\label{x1}(osp(3|n),n+1)=(spo(n|3),-\tfrac{n+1}{4})\triangleright \ldots\triangleright (spo(2|3),-\tfrac{3}{4})\triangleright (sl(2),1)\\
&\label{x3} (osp(5|n),\tfrac{n-1}{2})\triangleright(osp(1|n),\tfrac{-n-3}{4})=(spo(n|1),\tfrac{-n-3}{4})\triangleright \ldots\triangleright (spo(2|1),-\tfrac{5}{4})\triangleright\C \\
&\label{x5}(osp(7|n),\tfrac{n-3}{2})\triangleright(osp(3|n),1+n)\triangleright\eqref{x1}\end{align}\begin{align}
&\label{hh} m\text{ odd}, m\ge 8: \,\,(osp(m|n), \tfrac{n-m+4}{2})\triangleright  (osp(m-4|n),\tfrac{8-m+n}{2})\triangleright\ldots\triangleright \eqref{x3}\text{ or }\eqref{x5}
\\
&{\label{hhh} m\text{ even}, m>n+4, m\equiv n\mod4: \,\,(osp(m|n), \tfrac{n-m+4}{2})\triangleright  (osp(m-4|n),\tfrac{8-m+n}{2})\triangleright\ldots\triangleright\eqref{g}}\ \\
&{\label{hhhh} m\text{ even}, m>n+4, m\equiv n+2\mod4:}\\ & (osp(m|n), \tfrac{n-m+4}{2})\triangleright  (osp(m-4|n),\tfrac{8-m+n}{2})\triangleright\ldots\triangleright
(osp(n+6|n),-1)\triangleright (osp(n+2|n),1)\notag\\
&\label{o}(psl(m|m),-1)\triangleright\mathbb C
\\&\label{oo}m\ne n,n+1,n+2, m\ge 2, (sl(m|n),-1)\triangleright (M(1), 1)\\
&\label{qq} n\text{ odd }(sl(2|n),n/2-1)\triangleright(sl(n),-n/2) \triangleright\ldots \triangleright (sl(3),-3/2)\triangleright\mathbb C
\\ &n\text{ even } \label{r}\\
&\notag(sl(3|n),\tfrac{n-3}{2})\triangleright (sl(1|n),\tfrac{1-n}{2})= (sl(n|1),\tfrac{1-n}{2}) \triangleright   (sl(n-2|1),\tfrac{3-n}{2})   \triangleright   \cdots   \triangleright(sl(2|1),-\tfrac{1}{2})\triangleright\C \\
&\text{ $3< m,n, m$ even, $n$ odd }\label{tttt}\\
&\notag(sl(m|n),\tfrac{n-m}{2})\triangleright (sl(m-2|n),\tfrac{n-m}{2}+1)\triangleright \ldots\triangleright(sl(2|n),\tfrac{n}{2}-1)\triangleright\eqref{qq}\\
&\text{ $3< m,n, m$ odd, $n$ even }\label{u}\\\notag&(sl(m|n),\tfrac{n-m}{2})\triangleright (sl(m-2|n),\tfrac{n-m}{2}+1)\triangleright \ldots\triangleright(sl(3|n),\tfrac{n-3}{2})\triangleright\eqref{r}\\
& \text{ $3< n<m,\  m\equiv n\mod 2$ \label{t}}\\
&(sl(m|n),\tfrac{n-m}{2})\triangleright (sl(m-2|n),\tfrac{n-m}{2}+1)\triangleright \ldots\triangleright (sl(n+2|n),-1)=(sl(n|n+2),1)\notag\\
&\label{d21a1}(D(2,1;a),a)\triangleright (sl(2),-\tfrac{1+2a}{1+a})\\
&\label{d21a2}(D(2,1;a),-a-1)\triangleright (sl(2),-\tfrac{1+2a}{a})
\end{align}
\end{prop}
\begin{proof}The proof of the proposition is based on the  data in Table 2. This table is built up by using the data computed in 
\cite[Tables
5,6,7]{AKMPP-JA}.  Case \eqref{a} is proven by looking   at lines 19 and 
17 in Table 2. A similar direct analysis works in cases \eqref{b} (lines 21, 15), \eqref{c} (lines 23, 28), \eqref{l} (line 10), \eqref{m} (line 11), \eqref{e} (line 12),
\eqref{g} (line 16), \eqref{h} (line 15),   \eqref{x5} (line 13$'$), \eqref{hhh} and  \eqref{hhhh} (line 13),\eqref{o} (line 7), \eqref{oo} (line 6), \eqref{d21a1} (line 17) , \eqref{d21a2} (line 18).
\vskip5pt
\noindent{\bf Cases \eqref{f}}
Since $n$ is even and $m$ is odd,   $(m-n-2)/4$ is never a  integer, so $H(V_{(m-n-2)/4})\ne 0$; use line 8   up to arriving to $(spo(2|m),\tfrac{m-6}{4})$. Using line 8 prevents to consider the case $(spo(m-2,m),-\tfrac{1}{2})$: on the other hand  in this case we have  collapsing to $\C$ by line 12. If $m\geq 5$ by line 9 we arrive at $(so(m),\tfrac{4-m}{2})$ which collapses  according to Lemma  \ref{IMRNso}.
For $m=3$ we have $(spo(2|3),-\tfrac{3}{4})$ which collapses to $(sl(2),1)$ by line 10. For $m=1$ we have $(spo(2|1),-\tfrac{5}{4})$ which collapses to $\C$ by line 11.

vskip5pt
\noindent{\bf Case \eqref{x1}.}
Use  line  8 and \eqref{l}.
\vskip5pt
\noindent{\bf Case \eqref{x3}.}
Use   lines 14, 8 and \eqref{m}.
\vskip5pt
\noindent{\bf Case \eqref{hh}.}
Use line 13  up to arriving to \eqref{x3},  \eqref{x5}.
\vskip5pt
\noindent{\bf Case \eqref{qq}.} Use lines 4, 1, 3. Since $n$ is odd, all the levels involved are non-integral.
\vskip5pt
\noindent{\bf Case \eqref{r}.} Use lines 2, 1, 5. Since $n$ is even, all the levels involved are non-integral.
\vskip5pt
\noindent{\bf Cases \eqref{tttt}, \eqref{u}.} These cases are reduced to  \eqref{qq}, \eqref{r} by using \eqref{qq}, \eqref{r}.\vskip5pt
\noindent{\bf Cases \eqref{t}}  Use line 1 (note that the levels involved are negative, hence the functor $H$ is nonzero).

\vskip5pt
\end{proof}
\begin{table}\caption{}\label{Table2}
	
\noindent {\sl Values of $k$ and $k'$. Assume that $k \notin {\Z}_{\ge 0}$. 
\vskip 5pt

\centerline{\renewcommand{\arraystretch}{1.5}\begin{tabular}{c|c|c|c|c}
& $\g$&$V_{k'}(\g^\natural)$&$k$&$k'$\\
\hline
$1$& $sl(m|n)$, $m\neq n, m>3, m-2\ne n$&$V_{k'}(sl(m-2|n)))$&$\frac{n-m}{2}$&$\frac{n-m+2}{2}$\\ 
\hline
$2$& $sl(3|n)$, $n\ne3, n\ne 1, n\ne0$&$V_{k'}(sl(1|n)))$&$\frac{n-3}{2}$&$\frac{1-n}{2}$\\ 
\hline
$3$&$sl(3)$&$\C$&$-\frac{3}{2}$&$0$\\ 
\hline
$4$&$sl(2|n)$, $n\neq 2,n\ne1, n\ne 0$&$V_{k'}(sl(n)))$&$\frac{n-2}{2}$&$-\frac{n}{2}$\\ 
\hline
$5$&$sl(2|1)=spo(2|2)$&$\C$&$-\frac{1}{2}$&$0$\\ 
\hline
$6$&$sl(m|n)$, $m\neq n,n+1,n+2, m\ge2$&$M(1)$&$-1$&$1$\\ \hline
$7$&$psl(m|m),\,m\ge2$&$\C$& $-1$&$0$\\
\hline
$8$&$spo(n|m),m\ne n,n+2,n\ge 4$& $V_{k'}(spo(n-2|m))$ &$\frac{m-n-4}{4}$&$\frac{m-n-2}{4}$\\
\hline
$9$&$spo(2|m),m\ge  5$& $V_{k'}(so(m))$ &$\frac{m-6}{4}$&$\frac{4-
m}{2}$\\
\hline
$10$&$spo(2|3)$& $V_{k'}(sl(2))$ &$-\frac{3}{4}$&$1$\\
\hline
$11$&$spo(2|1)$& $\C$ &$-\frac{5}{4}$&$0$\\
\hline
$12$&$spo(n|m),m\ne n+1,n\ge 2$& $\C$ &$-1/2$&$0$\\
\hline
$13$&$osp(m|n),m\ne n,m\ne n+8,m\geq  8$&$V_{k'}(osp(m-4|n))$ &$\frac{n-m+4}{2}$&$\frac{8-m+n}{2}$
\\\hline
$13'$&$osp(7|n)$&$V_{k'}(osp(3|n))$ &$\frac{n-3}{2}$&$1+n$
\\\hline
$14$&$osp(m|n),n\ne m,0;4\le m\le 6$&$V_{k'}(osp(m-4|n))$ &$\frac{n-m+4}{2}$&$\frac{m-n-8}{4}$
\\\hline
$15$&$osp(m|n),m\ne n+4,n+8;m\geq 4$&$ V_{k'}(sl(2))$ &$-2$ &$\frac{m-n-8}{2}$\\
\hline
$16$&$osp(n+8|n),n\geq 0$&$\C$ &$-2$ &$0$\\
\hline
$17$&$D(2,1;a)$&$V_{k'}(sl(2))$&$a$ &$-\frac{1+2a}{1+a}$\\
\hline
$18$&$D(2,1;a)$&$V_{k'}(sl(2))$&$-a-1$ &$-\frac{1+2a}{a}$\\
\hline
 $19$&$F(4)$&$V_{k'}(D(2,1;2))$ &$-1$& $\frac{1}{2}$\\
\hline
$20$& $F(4)$&$\C$ &$-3/2$& $0$\\
\hline
$21$&$F(4)$&$V_{k'}(so(7))$&$\frac{2}{3}$ & $-2$\\\hline
$22$&$F(4)$&$\C$&$-\frac{2}{3}$ & $0$\\\hline
$23$&$G(3)$&$V_{k'}(G_2)$&$\frac{1}{2}$ & $-\frac{5}{3}$\\\hline
$24$&$G(3)$&$\C$&$-\frac{3}{4}$ & $0$\\\hline
$25$&$G(3)$&$V_{k'}(osp(3|2))$&$-\frac{2}{3}$ & $1$\\\hline
$26$&$G(3)$&$\C$&$-\frac{4}{3}$ & $0$\\\hline
$27$&$G_2$&$V_{k'}(sl(2))$&$-\frac{4}{3}$&$1$\\\hline
$28$&$G_2$&$\C$&$-\frac{5}{3}$&$0$
\end{tabular}}}

\end{table}

\begin{coro}\label{list} In cases \eqref{a}--\eqref{g}, \eqref{h} with $m-n\geq  5$,  \eqref{x1}--\eqref{t}  the category $KL^{fin}_k$ is semisimple.
If  $a \notin {\mathbb Q}$ or $a=\tfrac{q}{p}-1,p,q\in \mathbb Z_{\ge0},\, p\geq  1$ (resp. $a=-\tfrac{q}{p}$), the category $KL^{fin}_a$ (resp. $KL^{fin}_{-a-1}$) for $D(2,1; a)$ is semisimple.
\end{coro}
\begin{proof} We use Theorem \ref{maincoll}. In cases \eqref{a}, \eqref{b}, 
 the final level in  the collapsing chain is admissible.
In  case \eqref{h} the final vertex algebra is rational if $m-n\geq  7$ is even and admissible if it is odd. 
 The cases $m-n=5,6$ are covered by \cite[Theorem 4.4.1]{CY}. In case \eqref{d21a1}, the vertex algebra $V_{k'}(sl(2))$ is 
\begin{itemize}
\item[(1)] rational or admissible if $a=\tfrac{q}{p}-1$, $p,q\in {\mathbb Z}_{\ge 0}$, $p \ge 2$, and $KL_{k'}$ is semi-simple by \cite{AM95}, 

\item[(2)] isomorphic to $V_{-2+1/q}(sl(2))$ if $p=1$, and $KL_{k'}$ is semi-simple by  \cite[Theorem 4.4.1]{CY}, 
\item[(3)] generic if $a \notin {\mathbb Q}$,  and $KL_{k'}$ is semi-simple by  \cite{KL}. 
\end{itemize}

 Similarly for \eqref{d21a2}. In all other cases { except $V_1(osp(n+2|n), V_1(sl(n|n+2))$ the final vertex algebra  in  the collapsing chain is rational.
In  the two above special cases,  semisimplicity of $KL_k^{fin}$ is given by the following arguments:} 
\begin{itemize}
\item  Recall from \cite[Proposition 4.8]{AMPP-AIM} that the following decomposition holds:
$$ V_1(osp(2 m \vert n)) = V_1 (so(2m)) \otimes V_{-1/2} (sp(n)) +  V_1 (\omega_1) \otimes V_{-1/2} (\omega_1),$$
and the  embedding $V_1 (so(2m)) \otimes V_{-1/2} (sp(n))\hookrightarrow V_1(osp(2 m \vert n)) $ is conformal. Since 
$V_1 (osp(2m))$ is rational and $V_{-1/2} (sp(2n)) $ is admissible, we can apply Proposition \ref{klf=kl} (2) below to obtain $\sKlf = \Kl$. 

Assume that $W$ is any highest weight $V_1(osp(2 m \vert n))$--module in $\Kl$. Then $W$ is a completely reducible  as $V_1 (so(2m)) \otimes V_{-1/2} (sp(n))$ and it contains an irreducible $V_1 (so(2m)) \otimes V_{-1/2} (sp(n))$--submodule isomorphic to the exactly one of the following modules
\bea 
M_0 &\cong & V_1 (so(2m)) \otimes V_{-1/2} (sp(n)),\nonumber \\
M_1&\cong & V_1 (\omega_1) \otimes V_{-1/2} (\omega_1),\nonumber \\
M_2 &\cong & V_1 (\omega_1) \otimes V_{-1/2} (sp(n)),\nonumber \\
M_3&\cong & V_1 (so(2m)) \otimes V_{-1/2} (\omega_1). \nonumber 
\eea

Using the fusion rules arguments as in the proof of \cite[Proposition 4.8]{AMPP-AIM} we easily get that $W$ is isomorphic to exactly one of the following two irreducible modules
 \bea
 W_0  \cong M_0 \oplus M_1, \ W_1 = M_2 \oplus M_3. \nonumber  
\eea
 
Therefore 
 $V_1(osp(2 m \vert n))$  has two  irreducible modules in $KL_k$ and every highest weight module in $\Kl$ is irreducible. Now  using Theorem \ref{TT} we have that $\sKlf$ is semisimple.
 In particular,  for $V_1(osp(n+2|n))$, the category $\sKlf = \Kl$ 
 is semisimple.

 \item  $V_1(sl(n|n+2)) = V_{-1} (sl(n+2|n))$, and {  by results of Section \ref{slmn-1}} we have  semisimplicity in $\sKlf$.
\end{itemize}

\end{proof}

\section{Category  $\Kl$ of $\g$--locally finite $V_k(\g)$--modules}
\label{sect-5}

We first  investigate some sufficient conditions to have  either $KL_k^{ss} = \sKlf$ or $\Kl = \sKlf$.

\begin{prop} \label{klf=kl}
\begin{itemize}
\item[(1)]  Assume that $\g_{\overline 0}$ is a semisimple Lie algebra. Then
$\sKlf=KL_k ^{ss}$.
\item[(2)]  Assume that there    is a  conformal embedding of  $V_{k_1} (\g_{\overline 0}) \hookrightarrow V_k(\g)$ and every module $W$  from  $\Kl$ is semisimple as a $V_{k_1} (\g_{\overline 0})$--module. Then $\sKlf= \Kl$.
  \end{itemize}
\end{prop}
\begin{proof} 
Consider  case (1). Assume that $W$ is any module from $KL_k ^{ss}$. We need to show that $\h$  acts semisimply on $W$. 
 Each $L(0)$--eigenspace of $W$ is a sum of finite-dimensional $\g$--modules, therefore $W$ is a sum of finite-dimensional $\g$--module.
  Since there is an  embedding $V_{k_1} (\g_{\overline 0}) \hookrightarrow V_k(\g)$  and $\g_{\overline 0}$ is  semisimple,
  we conclude that $W$ is a direct sum of finite-dimensional  $\g_{\overline 0}$--modules. Since the action of the Cartan subalgebra $\h$ is obtained by the action of  operators from 
  $V_{k_1} (\g_{\overline 0}) $, we conclude that these operators act semisimply and therefore  $W$ is in $\sKlf$.

Now we consider the   case (2).  Let $W$ be a module from $\Kl$.
We have directly that $W$ is a semisimple as  $V_{k_1} (\g_{\overline 0})$--module.
 So $W$ is a direct sum of irreducible $V_{k_1} (\g_{\overline 0})$--modules in $KL_{k_1}$, which are highest weight modules, and 
therefore, since the embedding is conformal,  we get that 
 $\h$ and $L(0)$ must act semisimply. The claim follows.
\end{proof}
We have the following consequence:
\begin{coro} \label{largest}
Assume that the condition (2)   of Proposition \ref{klf=kl} holds and that
\begin{itemize}
\item[(3)] Any highest weight $V_k(\g)$--module in $\Kl$ is irreducible.
\end{itemize}
Then $\Kl$ is semisimple.
\end{coro}
\begin{proof}
The assumption  (2) implies  that the Cartan algebra and the Virasoro element $L(0)$ acts semisimply. This implies that $KL_{k} =  KL^{fin}_k$.   Then the result follows by applying Theorem \ref{TT}.
\end{proof}

The conformal embeddings $\g_{\overline 0}\hookrightarrow \g$ were classified in  \cite{AMPP-AIM}, and they  include all collapsing levels for Lie superalgebras. Only in some cases when $\g_{\overline 0}$ is reductive, the semisimplicity of $V_{k} (\g_{\overline 0})$ is still an open problem.

\begin{lem} \label{pom-1} Let  $M$ be a non-zero   $V_k(\g)$--module from $\Kl$. Then there is a  non-zero $V_k(\g)$--submodule  $ M^{fin} \subset M$ which belongs to $\sKlf$.  
  \end{lem}
  \begin{proof}
 Since $M_{top}$ is a locally finite $\g$--module,  we get  that $\mathfrak h$ and $L(0)$ acts locally finitely on $M_{top}$. So there is a common eigenvector $w$ for the action of $\mathfrak h$ and $L(0)$. Therefore $M^{fin} = V_{k} (\g). w$ is a $V_k(\g)$--submodule of $M$ which is in the category $\sKlf$.    \end{proof}

 \begin{lem}\label{lemma}
 \item[(1)]  Let $\mathcal M$ be  a logarithmic $V_k(\g)$--module in $\Kl$. Then $L(0) - L_{ss}(0)$ is a $V_k(\g)$--homomorphism.
  \item[(2)]  For any module $\mathcal M$  in  $\Kl$, the operator 
 $h(0) - h_{ss} (0)$ is a $V_k(\g)$--homomorphism  for any $h \in \frak h$.
 \end{lem} 
 \begin{proof}
The assertion (1) is already proved in \cite[Remark 2.21]{HLZ}. For completeness, we present here a version of their proof.
 
  We have
 $$ \mathcal M = \bigoplus _{\alpha \in {\Bbb C}} \mathcal M_{\alpha}, \ \mathcal M_{\alpha} = \{ v \in \mathcal M \ \vert \ (L(0)- \alpha) ^{N_\a} v =  0 \text{ for some $N_\a>0$}\}.$$
Define $Q \in \mbox{End}(\mathcal M)$  by 
$$Q v = (L(0)- \alpha )v, \quad v \in \mathcal M_{\alpha}. $$
Therefore $Q = L(0)-L_{ss} (0)$. Take $v \in \mathcal M_{\alpha}$ and $n \in {\Z}$. Then for each $N \in {\ZZ}_{>0}$ we have
$$ (L(0) - (\alpha-n) ) ^ N  x(n) v = x(n) (L(0) - \alpha ) ^ N v$$
which gives  that $x(n)v \in \mathcal M_{\alpha-n}$. Therefore
$$ x(n) Q v = x(n) L(0) v - \alpha x(n) v =   L(0) x(n)  v -  (\alpha-n)  x(n) v =   Q x(n) v,$$
which implies that $Q$ is a $\widehat {\g}$--homomorphism,  hence a  $V_k(\g)$--homomorphism. This proves (1). The proof of (2) is completely analogous.
 \end{proof}
 
    \begin{thm} \label{ext-kl-large} Assume that the category $\sKlf$ is semisimple and  that  for  any irreducible    $V_k(\g)$--module $M$   in $\Kl$ we have  
 \begin{equation}\label{Ext-1-2}  \mbox{Ext}^{1} (M_{top}, M_{top}) = \{0\}\end{equation}
  in the category of finite-dimensional $\g$--modules.
    Then $\Kl $ is semisimple and $\sKlf=\Kl$.
 \end{thm}
 
 \begin{proof}
 In a view of \cite[Lemma 1.3.1]{GK} it  suffices to show that 
 \begin{itemize}
 \item[(1)] $ \mbox{Ext}^{1} (M, N) = \{0\}$ for any two irreducible modules $M, N $ in $\Kl$;
 \item[(2)] Each module $M$ in  $\Kl$  contains an irreducible submodule.
 \end{itemize}
 Assume that we have a non-split  extension 
 \begin{equation}\label{Ext-1-3}  0 \rightarrow M \rightarrow M^{ext} \rightarrow N \rightarrow 0  \end{equation}
for a certain ${\Z}_{\ge 0}$--gradable module $M^{ext}$ in $\Kl$. 
If $M^{ext}$ is in $\sKlf$, then $M^{ext} \cong  M \oplus N$ because  $\sKlf$ is semisimple.   This contradicts the assumption that $M^{ext}$ is the non-split extension     (\ref{Ext-1-3}).\par  If $M^{ext}$ does not belong to  $\sKlf$, then 
  $L(0)$ does not act semisimply or there is $h \in \mathfrak h$ such that $h(0)$ does not act semisimply. 
  Define accordingly the operator $Q$ as in Lemma \ref{lemma}, i.e.  $Q = L(0) -L(0)_{ss}$ or $Q= h(0) - h(0)_{ss}$. 
  Then $Q$ is a  non-zero $V_k(\g)$--homomorphism and therefore $N \cong Q M^{ext}\cong M$. 
Indeed, since $M$ is irreducible, we have that  $QM=0$. It follows that $QM^{ext}\cong M^{ext}/Ker\,Q\subset M^{ext}/M\cong N$. Since $N\ne 0$, we find $QM^{ext}\cong N$. If the extension does not split, we must have $QM^{ext}\cap M \ne 0$, so $QM^{ext}\cong M$, and at the end $M\cong N$. By  applying the Zhu's functor to \eqref{Ext-1-3} we get a non-split   extension 
$$ 0 \rightarrow M_{top} \rightarrow (M^{ext})_{top} \rightarrow M_{top} \rightarrow 0$$
in the category of finite-dimensional $\g$--modules. This  contradicts \eqref{Ext-1-2}. 
 Thus (1) holds.
 
Let us prove (2). From Lemma \ref{pom-1} we get that $M$ contains a non-zero submodule $M^{fin}$ in $\sKlf$. Since $\sKlf$ is semisimple, we conclude that $M^{fin}$ contains an irreducible submodule. The claim follows.
 \end{proof}

 \section{ The case $\g=C(n+1)$  }

\subsection{Collapsing level $k=-\frac{1}{2}$.}

\begin{lem} Let 
 $\g=C(n+1)$ and $k=-\frac{1}{2}$.  Then the unique irreducible modules in $KL_{-\frac{1}{2}}$ are $V_{-\frac{1}{2}}(\g)$ and $L(-\frac{1}{2}\Lambda_0+\frac{1}{2}\theta)$, where 
 $\theta$ is the highest root of $sp(2n)$.
 \end{lem}
    \begin{proof} The discussion preceding   \cite[Lemma 4.1]{AKMPPIMRN}
    applies: if $L(\L)$ is irreducible in $KL_{-\frac{1}{2}}$ then $\L=-\frac{1}{2}\Lambda_0+\ell \theta$ and 
$\ell^2-(k+1)\ell=0$.
\end{proof}

\begin{thm} \label{cn1}  Let  $\g=C(n+1)$ and $k=-\frac{1}{2}$. Then we have:
\begin{itemize}
\item  Irreducible modules in $\Kl$ have no self-extensions.
\item The category $\sKlf$ is semisimple and $\Kl=\sKlf$.
\end{itemize}

\end{thm}

\begin{proof} 
It suffices to check that     condition (2)  of Proposition \ref{klf=kl}  and condition (3) of Corollary \ref{largest} hold.

Let us first check that:
\begin{itemize}
 \item[(*)] any $V_k(\g)$--module $W$  in $\Kl$ is  completely reducible as $V_{k}(\g_{\bar 0}) = V_{-1/2} (sp(2n)) \otimes M(1)$--module,
 where $M(1)$ is the Heisenberg vertex algebra of rank one.
 \end{itemize}
Let $V_D = M(1) \otimes {\C}[D]$ be the lattice vertex algebra associated to rank one lattice $D = \Z \alpha, \langle \alpha, \alpha \rangle = 4$. It has $4$ non-isomorphic modules:
$$ U_i = V_{ D + \tfrac{i }{4}\alpha}, i =0,1,2,3,   $$
and the following fusion rules: $$U_i \times U_j = U_{(i+j)\mod 4}. $$
From  \cite[Proposition 4.15]{AMPP-AIM}, we have  that
$$V_{-\frac{1}{2}}(\g) = (V_{-1/2} (sp(2n)) \otimes  U_0 )\bigoplus (L_{-1/2} (\omega_1) \otimes U_2) $$
 (note that in {\it loc. cit.} a different normalization is used, so that the level $1$ used there turns into level $-1/2$).
The irreducible module $L(-\tfrac{1}{2}\Lambda_0+\tfrac{1}{2}\theta)$ decomposes as
$$ L(-\tfrac{1}{2}\Lambda_0+\tfrac{1}{2}\theta) =  V_{-1/2} (sp(2n)) \otimes U_1   \bigoplus L_{-1/2} (\omega_1) \otimes  U_3.  $$

By using the regularity of $V_D$ (i.e., complete reducibility in the entire category of weak $V_D$--modules, cf. \cite{DLM-reg}) and  the concept of Heisenberg coset (cf. \cite{CKLR}), 
we get that
$$ W = M_0 \otimes U _0 \bigoplus M_1 \otimes U_1 \bigoplus M_2 \otimes U_2 \bigoplus M_3 \otimes U_3$$ 
for certain 
$V_{-1/2} (sp(2n))$--modules $M_i$  in $\Kl$, $i=0,1,2,3$.
By using  complete reducibility for the admissible vertex algebra $V_{-1/2} (sp(2n))$ in the category $\mathcal O$ (cf. \cite{A-1994}), we get that
the  $M_i$ are direct sum of copies of $V_{-1/2} (sp(2n)) $ or $L_{-1/2} (\omega_1)$.
This implies that $W$ is a direct sum of irreducible $V_{-1/2} (sp(2n)) \otimes  U_0$--modules. Since each irreducible $U_0$--module is a direct sum of irreducible modules for the Heisenberg vertex algebra $M(1)$, we get that $W$ is a direct sum of $V_{k}(\g_{\bar 0})$--modules. So (*) holds.

It remains  to prove that any highest weight $V_k(\g)$--module in $\Kl$ is irreducible.  
By using the same arguments as above and fusion rules, we see that  if $W$ is a highest weight module in $\Kl$, it  decomposes as
$$V_{-1/2} (sp(2n)) \otimes  U_0  \bigoplus L_{-1/2} (\omega_1) \otimes U_2  \quad \mbox{or} \quad V_{-1/2} (sp(2n)) \otimes U_1   \bigoplus L_{-1/2} (\omega_1) \otimes  U_3.  $$
 Therefore it is irreducible. The claim follows. 
\end{proof}

\subsection{Collapsing level $k=- \frac{n+1}{2}$.}
 
 \begin{thm}\label{newC} Assume that $\g = C(n+1)$, $k =-\frac{n+1}{2}$. Then   $\sKlf$ is semisimple.
\end{thm}
 \begin{proof}
First we consider the case $n=1$, so  $\g= C(2)\cong sl(2 \vert 1)$. Then $k=-1$ is the critical  level.  Recently T. Creutzig and J. Yang in \cite[Theorem 6.6]{CY} proved that in this case every module in $\sKlf$ is  completely reducible  
(to match Creutzig-Yang's result with our setting note that  $V_{-1}(sl(m|n)) =V_1(sl(n|m))$). Therefore we have the collapsing chain
$$ (C(n+1),- \frac{n+1}{2})\triangleright  (C(n),- \frac{n}{2}) \triangleright\ldots  \triangleright(C(2),-1)$$
implying that $\sKlf$ is semisimple. 
 \end{proof}

\section{ Semisimplicity of $KL_{-1}$ for $\g=\mathfrak{sl}(m \vert n)$ and $ \g = \mathfrak{psl}(n \vert n)$.}
In $sl(n|m), n\ne m$ set $\alpha_i^\vee=E_{ii}-E_{i+1i+1}$ for $i\ne n$ and $\alpha_n^\vee=E_{nn}+E_{n+1n+1}$ ($E_{ij}$ are matrix units). Define $\omega_i\in \h^*$ by setting $\omega_i(\alpha_j^\vee)=\delta_{ij}$ and $\omega_0=0$. When $n=m$, we work modulo the identity.\par
Recall that the defect ${\rm def}\,\g$ of a basic classical Lie superalgebra $\g$ is the dimension of a maximal isotropic subspace in the real span of roots. When $\g=sl(m|n)$
the defect is $\min\{m,n\}$. Also recall that the atypicality of a weight $\l$ is the maximal number of linearly independent mutually orthogonal isotropic roots which are also orthogonal to $\l$. The atypicality of an  irreducible finite dimensional   $\g$-module $V$ of  highest weight $\l$ is the atypicality of $\l+\rho$ (here $\rho$ is the half sum of positive even roots minus  the half sum of positive odd roots).
 
 In \cite{KWNT} it is shown that the atypicality does not depend on the choice of the set of positive roots.  Let $L_{\g_{\bar 0}}(\l)$ be the finite dimensional irreducible $\g_{\bar 0}$--module of highest weight $\l$. The following conditions are equivalent \cite{K0}:
 \begin{enumerate}
 \item $\l+\rho$ is atypical.
 \item The Kac module $Ind_{\g_{\bar 0}}^\g(L_{\g_{\bar 0}}(\l))$ is not irreducible.
 \end{enumerate}
 \noindent It turns out that the atypicality of the trivial module, i.e. that of $\rho$, is ${\rm def}\,\g$.
 \subsection{ The representation theory of  $V_1 (psl(m \vert m))$   via decomposition of conformal embedding
 }
 
 Let us consider first the case $V= V_1 (psl(m \vert m))$ for $m \ge 3$. 
For $s \in {\Z}_{\ge 0}$ consider the following $V_{-1}(sl(m))$-modules
$$\pi_s := L_{-1} (s\omega_1), \quad \pi_{-s}:= L_{-1} (s \omega_{m-1}). $$
In order to prove that $V$ is semisimple in $KL_1$, we need to prove the following:

\begin{prop} \label{prop-1}Assume that $M$ is any highest weight $V$--module in $KL_1$. Then $M \cong V$.
\end{prop}
\begin{proof}
Since $V_1(sl(m))$ is rational and    $KL_{-1}(sl(m))$ is semisimple \cite[Theorem 1.2]{AKMPPIMRN} we conclude that $M$ is  completely reducible as $U=V_1(sl(m)) \otimes V_{-1}(sl(m))$--module. This implies that $M$ contains a $U$--submodule isomorphic to
$U_{r_0,s_0} = L_1 (\omega_{r_0}) \otimes \pi_{-s_0}$, for some $0 \le r_0 \le n-1$, and $ s_0 \in {\Z}$. 
(For $r_0=0$, we  set $L_1 (\omega_0)  =V_1(sl(m))$.

The lowest conformal weight of  $U_{r,s}$  and $U_{r,-s}$ are
\bea  h[r,s]&=& \frac{ (\omega_r, \omega_r + 2 \rho)}{2 (m+1)} +  \frac{ ( s 
\omega_{m-1} , s \omega_{m-1} + 2 \rho)}{2 (m-1)} \nonumber \\
&=& \frac{(m-r) r}{2m} + \frac{s^2 + m s}{2m} =\frac{ (m-r)r + s(s+m)}{2m}.   \nonumber  \\
 &=& \frac{ (\omega_r, \omega_r + 2 \rho)}{2 (m+1)} +  \frac{ ( s 
\omega_{1} , s \omega_{1} + 2 \rho)}{2 (m-1)} = h[r,-s] \nonumber 
\eea

We can choose $(r_0,s_0)$ so that the conformal weight of $U_{r_0,s_0}$ coincides with the conformal weight of the highest weight vector of $M$.

For $s \in {\Z}$, let $\overline  s \in \{ 0, 1, \dots, m-1\}$  be such that 
$s \equiv \overline s \mod(m-1)$.
For $0 \le r_1, r_2 \le m-1$, 
we set 
$r_3= \overline{r_1+ r_2}$.
Recall   from \cite[Theorem 4.4]{AMPP-AIM} that  
$$ V =\bigoplus_{s\in {\Z}} U_{\overline s, s},$$
and that $V$ is generated by $V_{1}(sl(m)) \otimes V_{-1}(sl(m)) \oplus U_{1,1} \oplus U_{n-1,-1}$.

Using the  fusion rules:
\begin{align*}
 L_{1} (\omega_1) \times L_{1}(\omega_r) &= L_{1}(\omega_{\overline{r+1}}),\\
 L_{1} (\omega_{m-1}) \times L_{1}(\omega_r) &= L_{1}(\omega_{\overline{m-1-r}}),\\
 \pi_s \times \pi_{s'} &= \pi_{s+s'}, 
\end{align*}
we conclude that $M$ also contains submodules:
\bea &&  U_{ \overline{r_0+\ell} , s_0 +  \ell} \quad (\ell \in {\Z}).  \label{mod-lcw}\eea

Since $M$ is a highest weight module, it is generated by $U_{r_0,s_0}$  
$$M = V. U_{r_0,s_0} =\bigoplus_{\ell \in {\Z}} U_{ \overline{r_0+\ell} , s_0 +  \ell}. $$
We conclude that $M$ is  $\Z_{\ge 0}$--graded.
By a direct calculation of the lowest conformal weight of $U_{ \overline{r+\ell} , s +  \ell}$ we get that the following statements are equivalent:
\begin{itemize}
\item[(1)] $M$ is $\Z$--graded;
\item[(2)] $h[r+1,s+1]-h[r,s] = \frac{m -r + s}{m} \in {\Z}$ for $r\in \{0, 1, \dots, m-2\}$,  $s \in {\Z}$;
\item[] $h[r-1,s-1]-h[r,s] =- \frac{m -r + s}{m} \in {\Z}$ for $r\in \{1, 2, \dots, m-1\}$,  $s \in {\Z}$;
\item[(3)] $r\equiv s \mod n$.
\end{itemize}

But then  $M$ contains a $V_1(sl(m) ) \otimes  V_{ -1}(sl(m))$--submodule isomorphic to $U_{r, r}$, implying that  $M=V$.
The claim follows.
\end{proof}

Since $\g_{\bar 0} = sl(m) \times sl(m)$ is semisimple, and the categories $KL_1$ and $KL_{-1}$ are semisimple for $sl(m)$,  using Proposition \ref{prop-1} we conclude:
\begin{thm}\label{72} Assume that $\g =psl(m|m)$   for $m \ge 3$  and $k=-1$.   The category $\Kl$ is semisimple. 
 \end{thm}

\subsection{The case $V_{-1}(sl(m|1))$}\label{slmn-1} 
Let $\g = sl(m|1)$. 
Recall \cite{AMPP-AIM} that   there is a conformal embedding
$ V_{-1}(sl(m)) \otimes M_c(1) \hookrightarrow V_{-1}(\g) $
with the following decomposition:
$$V_{-1}(\g) =  \bigoplus_{q \in {\Z} } \pi_q  \otimes M_c (1, -q \sqrt{\tfrac{m-1}{m}} ). $$ 
where we set
$$ c=\frac{1}{\sqrt{m(m-1)}}\begin{pmatrix} I_m&0\\0&m\end{pmatrix}.
$$
so that
$ [c_{\lambda} c] = \lambda$.

For $\ell \in {\Z}$ and $r \in {\C}$, we define the following $V_{-1}(sl(m))\otimes M_c(1)$--module:
$$L[\ell, r] =  \bigoplus_{q \in {\Z} } \pi_ {q+\ell}   \otimes M_c(1,    -(q+r) \sqrt{\tfrac{m-1}{m}}).$$
 Note that  for each $s \in {\Z}$ we have
 $ L[\ell +s, r+s] = L[\ell, r]. $

Recall that $V_{-1}(sl(m \vert n) )$ is realized as a vertex subalgebra of $M_{m} \otimes F$, where $M_m $ is the Weyl vertex algebra with generators $a^{\pm} _i$, $i=1, \dots, m$, and $F_n $ is the Clifford vertex algebra with generators $\Psi^{\pm} _i $ (cf. \cite{KW}, \cite{AMPP-AIM}).

In this section we set $F=F_1$,  $\Psi^{\pm} = \Psi^{\pm} _1$. We consider $V_{-1}(\g)$ as subalgebra of $M_m \otimes F$.

\begin{prop}\label{prop-0} For every $\ell \in {\Z}$, $L[\ell, -\tfrac{  \ell  }{m-1}]$ has   the structure of an  irreducible $V_{-1}(\g)$--module. It is realized as 
$$ L[\ell, -\tfrac{\ell  }{m-1}] = V_{-1}(\g) w_{\ell}  $$
where $w_{\ell} = : ( a^+ _1)^{\ell-1}:\Psi^+$ for $\ell \in {\Z}_{> 0}$ and
$w_{\ell}  = : ( a^- _{m})^{\vert \ell \vert }:$ for $\ell \in {\Z}_{\le 0}$.

Moreover, $L[\ell, -\tfrac{\ell  }{m-1}]_{top} = U(\g). w_{\ell}$ is an {\sl atypical} irreducible, finite-dimensional $\g$--module.  
\end{prop}
\begin{proof}
The results from \cite{KW}  give that $M_m \otimes F$ is a completely reducible  $V_{-1}(gl(m \vert 1))$--module so that
$$ M_m \otimes F = \bigoplus _{\ell \in {\Z}} V_{-1}(gl(m \vert 1)). w_{\ell}. $$
This implies that $V_{-1} (\g) . w_{\ell}$ is an irreducible $V_{-1}(\g)$--module.  
Set 
$r  =  -\ell / (m-1)$.

By identifying the highest weights, we get
$$ V_{-1} (sl(m)) \otimes M_{c} (1) . w_{\ell} = \pi_{\ell} \otimes M_{c} (1, -r \sqrt{\tfrac{m-1}{m}}  ). $$
By using fusion rules for $ V_{-1} (sl(m)) \otimes M_{c} (1)$--modules we get: 
\bea
V_{-1} (\g) . w_{\ell} &=& V_{-1} (\g). \pi_{\ell} \otimes M_{c} (1,  -r \sqrt{\tfrac{m-1}{m}}  ) \nonumber \\
&=&  \bigoplus_{q \in {\Z} } ( \pi_q  \otimes M_c(1, -q \sqrt{\tfrac{m-1}{m}} ) ).  \pi_{\ell} \otimes M_{c} (1, -r  \sqrt{\tfrac{m-1}{m}}  )  \nonumber \\
&=&   \bigoplus_{q \in {\Z} } \pi_ {q+\ell}   \otimes M_c(1,    -(q+r) \sqrt{\tfrac{m-1}{m}} )  = L[\ell, r]. \nonumber  \eea
The top component is then an irreducible $\g$--module $U(\g). w_{\ell}$ which has all $1$--dimensional weight spaces. Therefore $U(\g). w_{\ell}$  can not isomorphic to the Kac module obtained from  the corresponding $gl(m)$--module. Therefore $U(\g). w_{\ell}$ is atypical. (In Remark \ref{atyp-calculation}  below we   check the atypicality by computing explicitly the highest weights).
The claim follows.
\end{proof}

The same argument of Proposition \ref{prop-1} yields:
\begin{prop} \label{prop-2}Let $m \ge 3$. Assume that $M$ is any highest weight $V_{-1}(\g)$--module in $KL_{-1}$.  Then $M$ is irreducible and there is $\ell \in {\Z}$ such that
 $$M  \cong L[\ell, -\tfrac{ \ell  }{m-1}]. $$
 In particular, $M_{top}$ is an  {\sl atypical} $\g$--module.  \end{prop}
\begin{proof}
Since $M$ is a highest weight $V_{-1}(\g)$--module in $\Kl$, its  top component $M_{top}$ must contain a singular vector $w$ such that 
$$ V_{-1} (sl(m)) \otimes M_c(1) . w \cong \pi_{\ell} \otimes M_c(1, -r \sqrt{\tfrac{m-1}{m}}  )$$
for certain $\ell \in {\Z}$ and $r \in {\C}$.   By using fusion rules for $ V_{-1} (sl(m)) \otimes M_{c} (1)$--modules we get: 
\bea
M  &=& V_{-1} (\g). \pi_{\ell} \otimes M_{c} (1,  -r \sqrt{\tfrac{m-1}{m}}  ) \nonumber \\
&=&   \bigoplus_{q \in {\Z} } \pi_ {q+\ell}   \otimes M_c(1,    -(q+r) \sqrt{\tfrac{m-1}{m}} )  = L[\ell, r]. \nonumber  \eea

Let
$h[\ell, r]$ denotes the conformal weight of the top component of $\pi_{\ell} \otimes 
  M_c(1,    -r \sqrt{\tfrac{m-1}{m}} ).$ It is given by the formula
  $$ h[\ell, r] = \frac{ \ell ^2 + \vert \ell \vert m}{2m} + r^2 \frac{m-1}{2m}. $$
  Since  the embedding $V_{-1}(sl(m)\otimes M_c(1)\hookrightarrow V_{-1}(sl(m|1))$ is conformal, the  conformal weight of $\pi_{q+\ell}\otimes M_c(1, -(q+r)\sqrt{\frac{m-1}{m}})$ must differ form 
  $h[\ell,r]$ by a positive integer. In particular,  we must have the following conditions:
\begin{itemize}
  \item $h[\ell -1, r-1] - h[\ell, r] \in {\Z}_{> 0}, $
   \item $h[\ell +1, r+1] - h[\ell, r] \in {\Z}_{\ge 0}. $
  \end{itemize}
These relations have the  solutions  $r = -\frac{\ell+m}{m-1}$ for $\ell \geq 0$ and
$r = -\frac{\ell}{m-1}$ for $\ell \le 0$ (the solution is indeed unique if $\ell\ne 0$). Therefore
\begin{align*}
M &\cong L[\ell,  -\frac{\ell+m}{m-1}] = L[\ell+1, -\frac{\ell +1}{m-1}] \quad (\ell \geq 0),\\
M &\cong L[\ell,  -\frac{\ell}{m-1}]  \quad (\ell \le 0). 
\end{align*}
The atypicality of $M_{top}$ follows from Proposition \ref{prop-0}.

\end{proof}
  \begin{rem}  \label{atyp-calculation} Less conceptually, we can check the atypicality by computing explicitly the highest weights.
  Identify $\h^*$ with $\{r_0\d_1+\sum_{i=1}^m r_i\varepsilon_i\mid r_0+\sum_{i=1}^mr_i=0\}$. Then
 the highest weight of $:(a^+_1)^\ell\Psi^+:$ ($\ell >0$) is  
\begin{align*}
& \l_\ell^+:=\ell \omega_1+(m+\ell)\left(-\frac{1}{m(m-1)}(\varepsilon_1+\cdots+\varepsilon_m)+\frac{1}{m-1}\d_1\right)
\\
&=\ell(\varepsilon_1-\frac{1}{m}\sum_{i=1}^m\varepsilon_i)-\frac{m+\ell}{m(m-1)}\sum_{i=1}^m\varepsilon_i+\frac{m+\ell}{m-1}\d_1\\
&=(\ell-\frac{\ell}{m}+\frac{\ell+m}{m-m^2})\epsilon_1+\sum_{i=2}^m\frac{1+\ell}{1-m}\varepsilon_i+\frac{m+\ell}{m-1}\d_1.
\end{align*}
 Since $
 \rho=\sum_{i=1}^m(\frac{m}{2}-i+1)\varepsilon_i-\frac{m}{2}\d_1
 $, we have 
\begin{align*}
\l_\ell^++\rho&=\frac{m\ell-2\ell-1}{m-1}\varepsilon_1+\sum_{i=1}^m(1-i+\frac{1+\ell}{1-m}+\frac{m}{2})\varepsilon_i+(\frac{m+\ell}{m-1}-\frac{m}{2})\d_1\\&=\frac{m\ell-2\ell-1}{m-1}\varepsilon_1+(\frac{m}{2}-\frac{m+\ell}{m-1})\varepsilon_2+\sum_{i=3}^m(1-i+\frac{1+\ell}{1-m}+\frac{m}{2})\varepsilon_i+(\frac{m+\ell}{m-1}-\frac{m}{2})\d_1,
\end{align*}
and $(\l_\ell^++\rho|\d_1-\varepsilon_2)=0$.
\par
 The highest weight of $:(a^-_m)^{-\ell}:$ is  
 \begin{align*}&\l^-_\ell:=-\ell \omega_{m-1}+\ell\left(-\frac{1}{m(m-1)}(\varepsilon_1+\cdots+\varepsilon_m)+\frac{1}{m-1}\d_1\right)\\
 &=\ell\left(-\sum_{i=1}^{m-1}\varepsilon_i+\frac{m-2}{m-1}\sum_{i=1}^{m}\varepsilon_i+\frac{1}{m-1}\d_1 \right)\\
 &=\ell\left(\frac{1}{1-m}\sum_{i=1}^{m-1}\varepsilon_i+\frac{m-2}{m-1}\varepsilon_m+\frac{1}{m-1}\d_1 \right).
 \end{align*}
 Hence
  \begin{align*}&\l^-_\ell+\rho=\ell\left(\frac{1}{1-m}\sum_{i=1}^{m-1}\varepsilon_i+\frac{m-2}{m-1}\varepsilon_m+\frac{1}{m-1}\d_1 \right)+\rho\\
 &=\sum_{i=1}^{m-1}(\frac{m}2-i+1-\frac{\ell}{m-1})\varepsilon_i+(1+\frac{(m-2)\ell}{m-1}-\frac{m}{2})\varepsilon_m+(\frac{\ell}{m-1}-\frac{m}{2})\d_1 
 \end{align*}
 and $(\l^-_\ell+\rho|\d_1-\varepsilon_1)=0$.
 \end{rem}

In the following theorem we need to use results 	from  \cite{Ger} and \cite{VDJ} (see also \cite{GS}), which we recall in our setting.  Let $\mathcal L^{(k)}$ be the category of finite dimensional  $sl(m \vert 1)$-modules on which the center acts with Jordan blocks of size at most $k$.
Let $\l,\mu$ be dominant weights, and let $\rho_1$ be the half sum of the positive odd roots. Combining \cite[Proposition 6.1.2. (iii)]{Ger}  and \cite[Lemma 6.6]{VDJ} one has 
\begin{prop}\label{G} If $\lambda$ has  atypicality $1$, then 
$$Ext_{\mathcal L^{(k)}}(L(\lambda),L(\mu))=\begin{cases} \C\quad&\text{if $\lambda=\mu\pm 2\rho_1$,}\\ 0 \quad&\text{otherwise.}\end{cases}$$
\end{prop}
 Note that the non-trivial extensions above were explicitly realized  inside  a  Kac module, which is a weight module. Since our category $\sKlf$ is semi-simple, such extensions cannot appear for $V_k(\g)$--modules. We shall now see that there are no non-trivial extensions in the larger  category $\Kl$.

\begin{thm} \label{complete-red-slm1} Assume that $\g = sl(m|1)$ for $m \ge 2$  and $k=-1$.  Let $M$ be an irreducible  $V_{-1} (\g)$--module  in  $KL_{-1}$. Then 
\begin{itemize}
\item[(1)] $ \mbox{Ext}^1 (M_{top}, M_{top}) = \{ 0 \} $
in the category of finite-dimensional $\g$--modules.
\item[(2)] $ \mbox{Ext}^1 (M, M) = \{ 0 \} $
in the category $KL_{-1}$.
\item[(3)] The category $\Kl$ is semisimple.
\end{itemize}
\end{thm}
\begin{proof}
The classification of irreducible modules in $\Kl$ implies that the top component $M_{top}$ of an irreducible module in $\Kl$ is an irreducible highest weight $\g$--module. Since ${\rm def}\,sl(m|1)=1$, the atypicality of $M_{top}$ is at most $1$. Then  assertion  (1)   follows from Proposition \ref{G}.
The assertions (2) and  (3) follow from (1) and semisimplicity of $\sKlf$  by using Theorem \ref{ext-kl-large}.
 \end{proof}

\begin{rem}\label{daCY}
Theorem \ref{complete-red-slm1} generalizes to the whole category $KL_{-1}(sl(m|1))$ the  semisimplicity result of Creutzig and Yang \cite{CY}, who deal with modules of finite length  with semisimple $\h$-action in $KL_{-1}$ for $V_{-1}(sl(m|n)) =V_1(sl(n|m))$ in the more general case $m\ge 2$, $n \ge 1$. Both results rely on the classification of irreducible modules in $KL_{-1}$, that we construct  using the conformal embedding 
$V_{-1}(\g)\hookrightarrow M_m\otimes F$, whereas  Creutzig and Yang use tensor categories and induced modules:  see \cite[Corollary 6.11]{CY}.
\end{rem} 
\section{The category $KL_k$ is not semisimple for $\g=sl(m \vert 1)$ and  $k\in {\Z}_{>0}$.}
\label{non-semisimple}

Theorem \ref{complete-red-slm1}  shows that indecomposable non-irreducible modules in $\Kl$ do not exist for $k=-1$.

 Using Zhu's algebra theory in \cite{GS}, the authors construct indecomposable weak $V_k (sl(m\vert 1))$--modules for $k=1$ on which the element $L (0)$ of the Virasoro algebra does not act semisimply (these modules are also  called logarithmic modules).
 Note that the level  $k=1$ is neither  conformal nor collapsing  for $\g= sl( m \vert 1)$.

 In this section we shall first refine the example presented in \cite{GS} and show that even smaller category $\sKlf$ is not semisimple for $k=1$. Next we shall extend this result for   $k\in {\Z}_{>0}$.
 
  Recall that the  vertex algebra $V_1(\g)$ is realized as a subalgebra of $M \otimes F_{m}$, where $M= M_1$ is the Weyl vertex algebra generated by $a^{\pm} = a^{\pm} _1$, and $F_{m}$ the Clifford vertex algebra generated by $\Psi  ^{\pm}   _i$, $i=1, \dots, m$ (cf. \cite{KW}).
 Even generators of  $V_1(\g)$ are realized by 
 $$ E_{i,j} := :\Psi^+_i \Psi^- _j:, \quad i,j =1, \dots, m, $$
 and odd generators by 
 $$ E_{1, j+1}:= :a^+ \Psi^- _j:, \  E_{ j+1,1}:= :a^- \Psi^+ _j:, \quad j=1, \dots, m. $$
 
 Define $\vert m\rangle = :\Psi_1 ^+ \cdots \Psi_m ^+: \in F_{m}$. We know from \cite{KW} that \bea :(a^+)^{\ell}: \otimes \vert m> \label{sing-ell} \eea
 is a singular vector for $V_1(\g)$ for each $\ell \in {\Z}_{\ge 0}$, and it generates an irreducible, highest weight $V_1(\g)$--module. In order to construct  indecomposable, highest weight modules, we need to allow that  $\ell$ in the formula (\ref{sing-ell}) is a negative integer in certain sense.  In order to achieve this we shall  consider  a  larger vertex algebra containing $M\otimes F_{m}$ such that formula (\ref{sing-ell}) makes sense for $\ell$ negative. Fortunately, there is a nice construction of the vertex algebra  $\Pi(0)$ obtained using a localisation of the Weyl vertex algebra $M$. The  vertex algebra $\Pi(0)$ was orginally constructed in \cite{BDT}, and  has appeared  recently in realisation of certain vertex algebras and their modules    \cite{efren}, \cite{A-CMP19},  \cite{AdP-2019}, \cite{AKR-2021}.
 
 Let $L= {\Z}c + {\Z} d$ be the rank two  lattice such that
 $$ \langle c, d \rangle = 2, \ \langle c, c \rangle = \langle d , d \rangle = 0. $$
 Let $V_L = M(1) \otimes {\C}[L]$ be the associated lattice vertex algebra (cf. \cite{K2}). Then $\Pi(0)$ is realized as the following subalgebra of $V_L$:
 $$ \Pi(0) = M(1) \otimes {\C}[\Z c]. $$
 There is an embedding of $M$ into $\Pi(0)$ such that
 $$ a^+ = e^{c}, \ a^- = - \frac{c(-1)  + d(-1)}{2} e^{-c}. $$
 Then $e^{-c}$ plays the role of the  inverse $(a^+)^{-1}$ of $a^+$.
 \begin{thm} \label{free-field-indecom} Define 
 $$\widetilde w := (a^+)^{-m} \otimes \vert m\rangle = e^{-m c} \otimes \vert m\rangle\in \Pi(0) \otimes F_m.$$ Then we have:
 \begin{itemize}
 \item $\widetilde W  = V_1(\g) \widetilde w$ is a highest weight  $V_1(\g)$--module in the category $\sKlf$.
 \item $\widetilde W$ is reducible and it contains a proper submodule isomorphic to $V_1(\g).$
 \end{itemize}
 In particular, the category $\sKlf$ is not semisimple for $k=1$.
 \end{thm}
 \begin{proof}
 The proof that $\widetilde w$ is a singular vector is the same as in the case of the singular vector (\ref{sing-ell}). Therefore $\widetilde W$ is an highest weight $V_k(\g)$--module. By using the action of $E_{1,j+1} (0)$, $j=1, \dots, m$ we see that 
 
 $$ \widetilde W_{top} = \mbox{span}_{\C} \{ e^{-s c} \otimes :\Psi^+ _{j_1} \dots  \Psi^+ _{j_s}:\}, $$
 for $0 \le s \le m$, $ 1 \le j_1 < \cdots < j_s \le m$. So
 $\dim  \widetilde W_{top} = 2^m$. This implies that $\widetilde W$  is in the category $\sKlf$. Moreover $\widetilde w$  is a highest weight vector for $\widetilde W$, hence $\widetilde W$ is indecomposable. Since ${\bf 1}_{M \otimes F_m} = e^0 \otimes {\bf 1}_{F_m} \in \widetilde W_{top}$, we conclude that
 $V_1(\g) \cong V_{1}(\g). {\bf 1}_{M \otimes F_m}$ is a proper submodule of  $\widetilde W$. Therefore, $\widetilde W$ is reducible and indecomposable $V_k(\g)$--module.
 \end{proof}
  
\begin{rem} Note that the building block for a construction of indecomposable $V_k(\g)$--module in $\Kl$ is the   indecomposable $M$--module $\Pi(0)$.  By using  the singular vectors $(a^+ _1)^{-m} \otimes\vert m\rangle$, 
we get indecomposable, weight   $V_1(sl(m \vert n) )$--modules for every $n \in {\Z}_{>0}$. But  one can  show that  these modules are in the category $\Kl$ if and only if $n=1$.  A more detailed analysis of these modules will appear elsewhere. \end{rem}

   Let now $k\in {\Z}_{>0}$ is arbitrary. In \cite[Corollary 5.4.3]{GS}, the authors proved that  $V_{k} (\g) = V^k(\g) / I$, where $I$ is the ideal in $V^{k} (\g)$ generated by the singular vector $e_{\theta}(-1) ^{k+1}{\bf 1}$.   Now we can combine this result with    Theorem \ref{free-field-indecom} and show that there exist indecomposable $V_{k}(\g)$--modules.
     
  \begin{coro} \label{kl-k-integer}
  The category $\sKlf$ is not semisimple for any $k \in {\Z}_{>0}$.
  \end{coro} 
\begin{proof}
 It is clear that there is a diagonal action of 
  $V^{k}(\g)$ on $V_1(\g) ^{\otimes k}$. Using \cite{GS},  one gets that
  $$ V_{k}(\g) \cong V^{k}(\g). (  \underbrace{ {\bf 1} \otimes \cdots \otimes {\bf 1}}_{k \  \mbox{times}}) \subset   V_1(\g) ^{\otimes k}. $$
  As a consequence, we have that $\widetilde W \otimes  V_1(\g) ^{\otimes (k-1)}$ is a $V_{k}(\g)$--module. Define
  \bea \widetilde w^{(k)} &=& \widetilde w  \otimes \underbrace{{\bf 1} \otimes \cdots \otimes  {\bf 1} }_{(k-1)  \  \mbox{times}}, \nonumber \\
 \widetilde W^{(k)} &=&  V_{k}(\g). \widetilde w^{(k)}  \subset   \widetilde W \otimes  V_1(\g) ^{\otimes (k-1)}.\nonumber \eea
  One easily sees that:
   $$ \widetilde W^{(k)} _{top} = \mbox{span}_{\C} \{ (e^{-s c} \otimes :\Psi^+ _{j_1} \dots  \Psi^+ _{j_s}:) \otimes \underbrace{{\bf 1} \otimes \cdots \otimes  {\bf 1} }_{(k-1)  \  \mbox{times}} \}, $$
 for $0 \le s \le m$, $ 1 \le j_1 < \cdots < j_s \le m$. We can then argue as in Theorem \ref{free-field-indecom} and conclude that $ \widetilde W^{(k)}$ is indecomposable with  $V_k(\g)\cdot( \underbrace{{\bf 1} \otimes \cdots \otimes  {\bf 1} }_{k  \  \mbox{times}})$ a proper  submodule.
\end{proof}

 \begin{rem} \label{analysis-adm}
Assume that for certain non-integral level $k \in {\C}$, we have \bea
V_{k}(\g) \hookrightarrow V_{k-1}(\g) \otimes V_1(\g). \label{cond-embed-new} \eea
 Then the same proof as above implies that  $\sKlf$ is not semisimple.
 Beyond integral  levels $k \ge 2$, we believe that (\ref{cond-embed-new}) holds for principal admissible levels (cf. \cite{KW4}, \cite{KW5}). As far as we know, this is not proved yet.  So, one expects that $V_k(\g)$ will be semisimple only for principal admissible levels $k$ such that $k-1$ is not admissible. In the case $\g= \mathfrak{sl}(2 \vert 1)$, such levels will be  described in Section \ref{9}.
\end{rem}

 \section{Semisimplicity of $V_{1/2}(psl(2\vert 2))$}
 \label{N4}
 
In this section we will present an example where $\sKlf$ is semisimple with $k$  non-collapsing and $W_k(\g, \theta)$ irrational.

 In this section   we let $k=1/2$, $k_0 =-3/2$.
Set  $V=W_{k} (psl(2\vert 2), \theta)$.  By \cite[\S 8.4]{KW2}, $V$ is isomorphic to the simple $N=4$ superconformal vertex algebra at central charge $c=-9$. A free-field realization of the vertex algebra $V$ was presented in \cite{A-TG}, where it was proved that there is a conformal embedding $V_{k_0} (sl(2)) \hookrightarrow V$. Next, we  consider the category of $V$--modules which belong to $KL _{k_0} (sl(2))$, which   we  denote   by $KL_{N=4} $.  Since $KL _{k_0} (sl(2))$ is semi-simple, $KL_{N=4} $ coincides with the category of ordinary $V$--modules.

Using the same methods as in Section \ref{sect-4} and the fact that for $\g = sl(2)$  we have $KL _{k_0} (sl(2)) = KL  ^{fin} _{k_0} (sl(2))  $ we get:
 \begin{lem} \label{self-n4} In $KL_{N=4 } $ we have:
 $$ Ext^{1} _{ KL_{N=4}} (V, V) = \{0 \}. $$
 \end{lem}

  \begin{thm} \ \label{psl22-n4}
  \begin{enumerate} 
  \item The category $KL_{N=4}$ is semisimple.  In particular, every highest weight $V$--module in $KL_{N=4 } $ is irreducible and isomorphic to $V$.
    \item The category $\sKlf$ for $V_{1/2} (psl(2 \vert 2))$ is semisimple.
    \end{enumerate}
  \end{thm}
\begin{proof}
  By \cite{A-TG}, $V$ is the unique irreducible $V$--module in $KL_{N=4 }$, Assume that $M$ is a highest weight module in $KL_{N=4}$.
  By using the relation $$ [e]([\omega] + 1/2) = 0 $$ in the Zhu's algebra $A(V)$ from \cite[Proposition 3]{A-TG}  (with notation used there), one  easily sees that $M$  is irreducible  and isomorphic  $V$. Now applying    Lemma \ref{self-n4}  we get that   $KL_{N=4}$ is semisimple.

  Assume now that $W$ is a non-zero  highest weight $V_{1/2} (psl(2 \vert 2))$--module in $\Kl$. Then
   $H(W)$ is a non-zero  highest weight $V$--module in $KL_{N=4}$ and therefore $H(W) \cong V$.
 Now, using Lemma \ref{kriterij}, we get that $\sKlf(psl(2 \vert 2))$ is semisimple.
  \end{proof}

 \section{ The category $\Kl$ for $sl(2 \vert 1)$  }\label{9}

 In this section let $\g= sl(2 \vert 1)$ and $k = -\frac{m+1}{m+2}$, $m  \in {\Z}_{\ge 0}$.
 Recall that $W_k(\g, \theta)$ is isomorphic to the $N=2$ superconformal vertex algebra, which is rational by \cite{A-2001}. Therefore $\sKlf$ is semisimple by 
 Theorem \ref{T4}. Let us see that in this case $\Kl=\sKlf$.
 
 \begin{thm} \label{semisl21} Let $\g= sl(2 \vert 1)$ and $k = -\frac{m+1}{m+2}$, $m  \in {\Z}_{\ge 0}$. Then $ KL_k$ is semisimple. 
 \end{thm}
      \begin{proof}  
  Note that  $V_k(\g)$ has a subalgebra isomorphic to the affine vertex algebra $\widetilde V_{k}(sl(2))$, which is a certain quotient of $V^{k}(sl(2))$. 
Theorem \ref{com-dm1} below  implies  that  there is a conformal embedding  $U =\widetilde V_{k}(sl(2))\otimes W \hookrightarrow V_k(\g)$, where $W$ is isomorphic to the  regular vertex algebra $D_{m+1,2}$ from \cite{ACEJM}, also investigated recently in  \cite{YY1, YY2}.   In the case $m=0$, the assertion is already proved in Theorem \ref{cn1} and the vertex algebra $D_{1,2} $ is the lattice vertex algebra $V_D$.  Since $\g_{\bar 0}\cong sl(2)\oplus \C H^+$, and the center  $\C H^+$  belongs to $W$ (see \eqref{zincenter}), we have that $V_k(\g_{\bar 0}) \subset U$.  Note that the  $sl(2)$ subalgebra of $\g_{\bar 0}$ acts semisimply  on any module $M$ from $KL_k ^{ss}$. Since   $H^+\in  W$ and $W$ is regular,  we get that $M$ is  completely reducible as a module for the Heisenberg vertex algebra generated by $H^+$, which we   denote by $M_{H^+}(1)$.  Therefore, $M$ is a sum of $\widetilde V_k(sl(2)) \otimes M_{H^+}(1)$--modules in 
  $KL_k ^{ss} $ with semisimple action of $sl(2)$.   This implies that  the Cartan subalgebra of $\g$  acts semisimply on any module in $Kl_k ^{ss}$.  
 Hence $ KL_k ^{ss} = \sKlf. $

Assume $M$ and $N$ are irreducible $V_k(\g)$--module in $\Kl$ and that we have an   extension 
\bea  0 \rightarrow M \rightarrow M^{ext} \rightarrow N \rightarrow 0 \label{ext2} \eea
for a certain ${\Bbb  Z}_{\ge 0}$--gradable module $M^{ext}$ in $\Kl$. 
We have the following cases:
\begin{itemize}
\item If $M^{ext}$ is in $ KL_k ^{ss} =\sKlf$, then $M^{ext} \cong  M \oplus N$ because  $\sKlf$ is semisimple.
\item If $M^{ext}$ is  logarithmic,  then $Q = L(0)- L_{ss}(0)$ is a $V_k(\g)$--homomorphism and therefore $N \cong Q M \cong M$. So we can assume that $M=N$.
Then applying Zhu's functor we get an  extension 
$$ 0 \rightarrow M_{top} \rightarrow (M^{ext})_{top} \rightarrow M_{top} \rightarrow 0.$$
We proved above that  $H^+(0)$ is semisimple.
Note that $(M^{ext})_{top}$ is finite-dimensional and therefore the action of $sl(2)$ on  $(M^{ext})_{top}$  is semi-simple.  Thus the Cartan subalgebra $\h$ of $\g$ acts diagonally  on $(M^{ext})_{top}$.
Moreover, the action of $L(0)$  on  $(M^{ext})_{top}$ is given by
$$ L(0) = L^{sl(2)}  (0) + L^{W} (0)$$
where $L^{sl(2)}$ is Sugawara Virasoro vector in $\widetilde  V_k(sl(2))$, and $L^{W}$ is Virasoro vector in $W$. Since $W$ is regular, we have that the action of $L^{W}(0)$ is diagonal. Since the action of $L^{sl(2)}  (0) $ on  $(M^{ext})_{top}$  is proportional to the action of Casimir element of $sl(2)$,  we conclude that $L^{sl(2)}  (0) $ also acts diagonally on $(M^{ext})_{top}$. Therefore $L(0)$ acts diagonally on $(M^{ext})_{top}$. This implies that $M^{ext}$ is a module from $\sKlf$, and therefore (\ref{ext2}) splits.
\end{itemize}

  \end{proof}

  \begin{rem} Using the language of tensor categories and concepts from \cite{CY}, Theorem \ref{semisl21} implies    that $\Kl$ is a  semisimple braided tensor category.
  \end{rem}

 \begin{rem} 
 Although $KL_k(sl(2))$ is semi-simple for $k$ admissible, since $V^k(sl(2))$ is not simple, the category  $KL^k(sl(2))$ is not semi-simple.  Moreover,  it is expected that $KL^k(sl(2))$ contains logarithmic modules. The paper \cite{Ras} presents some conjectural logarithmic modules at admissible level, which  to belong  $KL^k(sl(2))$. \footnote{We thank the referee for this information.}
 
 We believe  that the subalgebra  $\widetilde V_k (sl(2))$  of $V_k(\g)$ is simple for $k=-\frac{m+1}{m+2}$. But we don't need this information for proving semisimplicity of $\Kl$.
 \end{rem}
  
   Based on Theorem \ref{semisl21} and  the arguments presented  in Remark \ref{analysis-adm}  we expect that the following conjecture holds.
   
   \begin{conj}\label{conj-s-s-21}
   The category $\Kl$ is semisimple if and only if $k \in 
   \{ -1, -\frac{m+1}{m+2} \ \vert  \ m \in {\Z}_{\ge 0}\}. $
   \end{conj}
       
\subsection{The vertex algebra $D_{m+1,2}$.   }

The vertex algebras $D_{m+1,k}$ are defined in \cite{ACEJM} for arbitrary $m \in {\C}$ and $ k\in {\Z}_{>0}$. We shall here assume that $m \in {\Z}_{\ge 0}$ and $k=2$.
 
For $p \in {\Z}$, let $F_{p} = M_{\delta} (1) \otimes {\C}[\Z \delta]$ be the rank one lattice vertex algebra associated to the lattice ${\Z}\delta^{(p)}$,
$\langle \delta^{(p)} , \delta^{(p)}  \rangle = p$. Here $M_{\delta^{(p)}}(1)$ is the Heisenberg vertex algebra generated by $\delta^{(p)} (z) = \sum_{n \in {\Z}} \delta^{(p)} (z) z^{-n-1}$.
Following \cite{ACEJM}, let 
$D_{m+1,2}$ be  the  vertex subalgebra of $V_{m+1}(sl(2)) \otimes F_2$ generated by
$$ \overline X =e(-1){\bf 1} \otimes e^{\delta^{(2)} }, \ \  \overline Y =f(-1){\bf 1} \otimes e^{-\delta^{(2)} }.$$

It was proved in \cite{ACEJM} that $D_{m+1,2}$ is a regular vertex operator algebra (i.e., every weak $D_{m+1,2}$--module is  completely reducible) and that
$$ D_{m+1,2} \otimes F_{-2} = V_{m+1} (sl(2)) \otimes F_{-2(m+2)}. $$

Consider now the vertex algebra $V_k(\g)$ for $k=-\frac{m+1}{m+2}$ and $\g= sl(2 \vert 1)$.   It is generated by four even fields $E^{1,2}, H^-, F^{1,2}, H^+$, such that 
$E^{1,2}, 2 H^-, F^{1,2}$ define the homomorphism $\Phi_1: V^k(sl(2)) \rightarrow  V_k(\g)$ commuting with the Heisenberg subalgebra generated by $H^+$, and four odd fields
$E^1, E^2, F^1, F^2$ (cf. \cite{BFST}). 
Denote by
$$ X= E^1 (-1) F^2 (-1) {\bf 1}, \quad Y = F^1 (-1) E^2(-1) {\bf 1}. $$
Set $X(n) = X_{n+1}, \ Y(n) = Y_{n+1}$. 
Let $W$ be the vertex subalgebra of $V_k (\g)$ generated by  vectors $X$ and $Y$. A direct computation, which uses the relations displayed in 
\cite[Appendix D]{BFST}, shows that 
\begin{equation}\label{zincenter}
X(1)Y=k\vac+2H^+.
\end{equation}

\begin{lem} \label{lem-rel-1}In $V_k (\g)$, we have:
\bea  X(- 2 m-4) X(-2 m-2) \cdots X(-2) {\bf 1} &=&  0, \label{rel-01} \\
X(- 2 m-5) X(-2 m-2) \cdots X(-2) {\bf 1} &=&  0. \label{rel-02} 
\eea
\end{lem}
\begin{proof}
The proof the relation (\ref{rel-01})  follows from the fact that there is $\nu \ne 0$ such that
$$ X(-2 m-4) X(-2 m-2) \cdots X(-4) X(-2) {\bf 1}\!=\! \nu F^2(-m-2) \cdots F^2 (-1) E^1 (-m-2) \cdots E^1 (-1) {\bf 1},$$and
  $$ F^2(-m-2) \cdots F^2 (-1) E^1 (-m-2) \cdots E^1 (-1) {\bf 1}$$
is a charged singular vector in $V^k(\g)$ (cf. \cite{ST}). The relation (\ref{rel-02}) follows by applying derivation on (\ref{rel-01}).
\end{proof}

Consider the vertex algebra $V_k (\g) \otimes F_{-2}$. Set $\delta = \delta^{(-2)}$. As shown in \cite[(4.2)]{BFST},  there exists   another  vertex algebra homomorphism $\Phi_2: V^{m+1}(sl(2)) \rightarrow  W \otimes F_{-2} \subset V_k(\g) \otimes F_{-2} $ uniquely determined by 
$$ e\mapsto  \frac{1}{k+1} X \otimes e^{\delta}, \  f\mapsto \frac{1}{k+1} Y \otimes e^{-\delta}, \quad h \mapsto (m+1) \delta + 2 (m+2) H^+.$$
  Using  Lemma \ref{lem-rel-1} we get that $e(-1) ^{m+2}{\bf 1} = 0$ in $W \otimes F_{-2}$, implying that
  $$ V_{m+1}(sl(2)) \hookrightarrow W \otimes F_{-2}.$$
 
 \begin{thm} \label{com-dm1}We have $W \cong D_{m+1,2}$.
 \end{thm}
 \begin{proof} Let $\beta =  (m+2) ( \delta +  2  H^+)$.   Then $\beta \in \mbox{Com} (V_{m+1} (sl(2)),  W \otimes F_{-2})$ and 
  $\beta(1) \beta =   -2 (m+2). $ Let $M_{\beta}(1)$ be the Heisenberg vertex algebra generated by $\beta$.
  Let $$u^+ = X(-2m-2) \cdots X(-2) {\bf 1} \otimes e^{(m+2) \delta}, \ \ u^- = Y(-2m-2) \cdots Y(-2) {\bf 1} \otimes e^{-(m+2) \delta} $$ Then
  $$ \beta(0) u ^+ = (2  (m+2) (m+1) - 2(m+2) ^2 ) u^+ = - 2(m+2) u ^+, \ \beta(0) u ^- = 2 (m+2)  u^-.$$
   
  We have:
  \bea  (k+1) e(0) u^+ &=&  X(-2 m-6)   X(-2m-2) \cdots X(-2) {\bf 1}  \otimes e^{\delta}_{2m+4} e^{(m+2)\delta}. \nonumber \\
  && + X(-2 m-5)   X(-2m-2) \cdots X(-2) {\bf 1}  \otimes e^{\delta}_{2m+3} e^{(m+2)\delta} = 0 \nonumber \eea
  Since $Y(2 m+3 + j)   X(-2m-2) \cdots X(-2) {\bf 1} = 0$ for $j \ge 0$, we conclude 
  \bea  (k+1) f(0) u^+ &=&  Y(2 m+3)   X(-2m-2) \cdots X(-2) {\bf 1}  \otimes e^{-\delta}_{-2m-5} e^{(m+2)\delta}= 0 \nonumber \eea
  We prove analogous relations for $u^- $, which implies that
  \bea  u^{\pm} \in \mbox{Com} (V_{m+1} (sl(2)),  W \otimes F_{-2}). \label{izom-upm}
  \eea

  Let $Z$ be the subalgebra of  $W \otimes F_{-2}$ generated by $u ^{\pm}$.
  So we have that $V_{m+1}(sl(2)) \otimes Z \subset  W \otimes F_{-2}$. By applying the action of $f(n)$ on $u^+$ (resp. $e(n)$ of $u^-$), one gets that $e^{\pm \delta} \in V_{m+1}(sl(2)) \otimes Z$. From this one gets $X, Y \in V_{m+1}(sl(2)) \otimes Z$. Therefore:
  $$ W \otimes F_{-2} = V_{m+1}(sl(2)) \otimes Z. $$
  
  We conclude that there is a conformal embedding $M_{\beta}(1) \hookrightarrow Z$, so that $Z$ is generated by singular vectors $u^{\pm}$. Using results on the uniqueness of the lattice vertex algebras (cf. \cite{LX}), we get that  
  $Z \cong F_{-2(m+2)}$. The isomorphism is determined by
  $$ e^{\delta^{(-2 (m+2) )}} \mapsto u ^+, \quad e^{-\delta^{(-2 (m+2) )}} \mapsto a u ^- , $$
  for a certain $a \ne 0$.
  Therefore we get an isomorphism
$$ W \otimes F_{-2} \cong V_{m+1}(sl(2)) \otimes F_{-2(m+2)}.$$ 
     On the other hand, using construction from \cite[Section 6]{ACEJM} we get an isomorphism
 \bea V_{m+1}(sl(2)) \otimes F_{-2(m+2)} \cong D_{m+1,2} \otimes F_{-2} \label{isom-dm2} \eea
 such that the subalgebra $F_{-2}$ is generated by 
 \bea  f(-1) ^{m+1}{\bf 1}  \otimes  e^{\delta^{(-2 (m+2) )}}, \quad   e(-1) ^{m+1} {\bf 1} \otimes  e^{-\delta^{(-2 (m+2) )}}. \label{gen-d2-new} \eea
 (This argument is essentially  based  on the fact that  $f(-1) ^{m+1}{\bf 1}, e(-1) ^{m+1} {\bf 1}$ generate a subalgebra of $V_{m+1}(sl(2))$ isomorphic to $F_{2m+2}$.)
The isomorphism (\ref{isom-dm2}) maps  the generators (\ref{gen-d2-new})  to elements of  $W \otimes F_{-2} $:
$$f(-1) ^{m+1} u^+, \  a e(-1) ^{m+1} u^-,$$
which are proportional to 
$e^{\delta}$, $e^{-\delta}$. Thus, we get an isomorphism $W \otimes F_{-2} \cong D_{m+1,2} \otimes F_{-2}$ which preserves $F_{-2}$.
    This   implies that $W \cong D_{m+1,2}$.

  \end{proof}

\vskip20pt
 \footnotesize{
  \noindent{\bf D.A.}:  Department of Mathematics, Faculty of Science, University of Zagreb, Bijeni\v{c}ka 30, 10 000 Zagreb, Croatia;
{\tt adamovic@math.hr}

\noindent{\bf P.MF.}: Politecnico di Milano, Polo regionale di Como,
Via Anzani 42, 22100 Como,\newline
Italy; {\tt pierluigi.moseneder@polimi.it}

\noindent{\bf P.P.}: Dipartimento di Matematica, Sapienza Universit\`a di Roma, P.le A. Moro 2,
00185, Roma, Italy;\newline {\tt papi@mat.uniroma1.it}
}

\end{document}